\documentclass{amsart}

\usepackage{amssymb}

\usepackage{amsmath}

\usepackage{tikz}

\usepackage{url}

\usepackage{apptools}
 
\usepackage{ulem}
\usepackage{url}

\usepackage{apptools}
\usepackage{graphicx} 
\usepackage[margin=2cm]{geometry}
\usepackage[initials]{amsrefs}
\usepackage{amsmath}
\usepackage{amsfonts}
\usepackage{amsthm}
\usepackage{amssymb}
\usepackage{enumitem}
\usepackage{xcolor}

\usepackage{todonotes}

\usepackage{comment}

\usepackage{todonotes}

\usepackage{comment} 
\newtheorem{theorem}{Theorem}[section]
\newtheorem{lemma}[theorem]{Lemma}

\newtheorem{claim}[theorem]{Claim}
\newtheorem{proposition}[theorem]{Proposition}
\newtheorem{fact}[theorem]{Fact}
\newtheorem{Corollary}[theorem]{Corollary}
\theoremstyle{definition}
\newtheorem{definition}[theorem]{Definition}

\theoremstyle{remark}

\numberwithin{equation}{section}

\begin{document}

\title[Isomorphism of pointed minimal systems]{ Isomorphism of pointed minimal systems is not classifiable by countable structures}


%    author one information

\author{Ruiwen Li}
\address{Nankai university}
\curraddr{School of Mathematical Sciences and LPMC, Nankai University, Tianjin 300071, P.R. China}
\email{rwli@mail.nankai.edu.cn}

\author{Bo Peng}
\address{McGill university}
\curraddr{Department of Mathmatics and Statistics, McGill University. 805 Sherbrooke Street West Montreal, Quebec, Canada, H3A 2K6}
\email{bo.peng3@mail.mcgill.ca}
\thanks{}

%    \subjclass is required.
\subjclass[]{}

\date{}

\dedicatory{}

%    "Communicated by" -- provide editor's name; required.
\commby{}

%    Abstract is required.
\begin{abstract}
 We prove that the topological conjugacy relations both for minimal compact systems and pointed minimal compact systems are not Borel-reducible to any Borel $S_{\infty}$-action.
\end{abstract}

\maketitle


\bibliographystyle{amsplain}
\section{Introduction}

      Over the last thirty years, a descriptive set-theoretic complexity theory for equivalence relations defined on standard Borel spaces has been developed by Kechris, Louveau, Hjorth and others (see \cite{Ke} and \cite{Hjorth}).  Let $X$ and $Y$ be two Polish spaces. For equivalence relations $E$ on $X$ and $F$ on $Y$, if there exists a Borel map from $X$ to $Y$ such that $f(x)Ff(y)$ iff $xEy$, then we say $E$ is \textbf{Borel reducible} to $F$, denoted by $E\le_B F$. We write $E\sim_B F$ if both $E\leq_B F$ and $F\leq_B E$ hold. For two equivalence relations $E$ and $F$, if $E$ is Borel reducible to $F$, then we will regard $F$ as a more complicated equivalence relation. 
      
      However, given an arbitrary equivalence relation, it is hard to find the exact complexity of it among all equivalence relations. Some equivalence relations play a more important role in the complexity theory. We call them benchmarks. By comparing an equivalence relation with those benchmarks, we can have a better understanding of the complexity of the given equivalence relation. 

     The group $S_{\infty}$ is the group of permutations of $\mathbb{N}$ with the topology of pointwise convergence. We say that an equivalence relation is \textbf{classifiable by countable structures} if it is Borel reducible to an equivalence relation induced by a Borel $S_{\infty}$-action. An equivalence relation $E$ is \textbf{Borel-complete} if it is classifiable by countable structures and every equivalence relation which is classifiable by countable structures is Borel-reducible to $E$. For group action cases, Hjorth \cite{Hjorth} introduced the notion of \textbf{turbulent actions} and showed that equivalence relations induced by turbulent actions are not classifiable by countable structures. In dynamical systems, Foreman and Weiss \cite{FW} proved that the conjugacy action of the group of measure preserving transformations on itself is turbulent. Carmelo and Gao \cite{CG} proved that the conjugacy relation of Cantor systems is Borel-complete. Also, Farah, Toms and T$\ddot{{\rm o}}$rniquist \cite{FTT} proved that the isomorphism of separable simple AI algebras is not classifiable by countable structures. The complexity of conjugacy relation of Cantor minimal systems and ergodic measure preserving transformations are still open. A famous result of Foreman, Rudolph and Weiss \cite{FRW} shows that the conjugacy relation of all ergodic measure preserving transformations is not Borel. Based on their construction, a very recent result of Deka, Garc\'{i}a-Ramos, Kasprzak, Kunde and Kwietniak \cite{DGKKK} shows that the conjugacy relation of Cantor minimal systems is not Borel. 

      An equivalence relation is called an \textbf{orbit equivalence relation} if it can be induced by a Borel action of a Polish group. A complete orbit equivalence relation is an equivalence relation induced by a Polish group action such that all orbit equivalence relations are Borel reducible to it. A famous result of Sabok \cite{Sabok} shows that the affine homeomorphism relation of Choquet simplices is a complete orbit equivalence relation. Based on this result, Zielinski \cite{Zielinski} proved that even the affine homeomorphism of Bauer simplices is a complete orbit equivalence relation, see  \cite{CG2, Ciesla} for related results. As a direct consequence of this result, the conjugacy relation of compact systems is also complete, see \cite[Page 92]{Sabok1} for a simple reduction. In fact, Bruin and Vejnar \cite{BH} showed that the conjugacy relation of Hilbert cube systems is a complete orbit equivalence relation, see also \cite{KV} for a related result.
      
      There are two related results in topological dynamics. Williams \cite{Williams} proved that all Bauer simplices can be realized as the set of invariant measures of minimal compact systems and studied the invariant measure of \textbf{Toeplitz subshifts}. Downarowicz \cite{Downarowicz} generalized this result to all Choquet simplices and showed that all Choquet simplices can be realized as the set of invariant measures of Toplitz symbolic subshifts. In fact, it remains unclear what the complexity of the isomorphism relation on the Toeplitz subshifts is \cite[Question 1.3]{ST}. In the context of symbolic subshifts, the conjugacy relation is a countable equivalence relation, so it can not be a complete orbit equivalence relation. However, Williams's construction (Oxtoby subshifts in $\Sigma^{\mathbb{Z}}$ with an arbitrary compact metric space) \cite{Williams} gives arbitrary minimal compact systems. A natural question is the following: can we map a compact metric space $X$ to a minimal compact system whose set of invariant measures is $P(X)$ (all probablity measure supported on $X$ with weak* topology) in a Borel way such that the isomorphism type of this system is determined by the homeomorphism type of $X$?  This would imply that the conjugacy relation of minimal compact systems is complete in all orbit equivalence relations. However, as it turns out (see Theorem 5.1 and Theorem 7.2), the isomorphism of Oxtoby systems is a Borel equivalence relation so it can not be a complete orbit equivalence relation.

      Another interesting problem is the conjugacy of pointed minimal compact systems. Kaya \cite{Kaya} proved that the conjugacy relation for pointed Cantor minimal systems is Borel bireducible with $=^{+}_{\mathbb{R}}$ which stays quite low in the Borel reducibility hierarchy. It is natural to ask about the complexity for general pointed minimal compact systems. Our main result shows that in the setting of all minimal compact systems the complexity grows considerably:
      
      \begin{theorem}\label{PCM}
          The conjugacy relation of pointed minimal compact systems is not Borel-reducible to any Borel $S_{\infty}$-action.
      \end{theorem}
          
     As consequences of our construction, we get the following,

      \begin{theorem}\label{CM}
          The conjugacy relation of minimal compact systems is not Borel-reducible to any Borel $S_{\infty}$-action.
      \end{theorem}

      \begin{theorem}\label{FCM}
          The flip conjugacy relation of minimal compact systems is not Borel-reducible to any Borel $S_{\infty}$-action.
      \end{theorem}

      To get those results, we study the isomorphism relation restricted to Oxtoby systems which will be defined in Section \ref{Toepliz}. The isomorphism type of Oxtoby systems is intrinsically connected with \textbf{topological type} of sequences which will be introduced in Section \ref{TT relation}. We will map a sequence in the Hilbert cube to an Oxtoby system such that two sequences have the same topological type if and only if the associated Oxtoby systems are conjugate. The main technical part is Theorem \ref{Main1} and we provide two different proofs to that theorem in Section \ref{proof 1} and \ref{proof 2}. In Section \ref{Borel reduction}, we show that one can reduce $c_0$ to topological type relation of sequences, which gives us Theorem \ref{PCM}, Theorem \ref{CM} and Theorem \ref{FCM}. Then we talk about some estimations about the complexity of the topological type relation in Section \ref{Ding Gu} and Section \ref{s9}.
      
\section{Preliminaries}

     In this paper, we will use the following notation. Let $X$ be a set. For an interval $I=[a,b]$ with $a,b\in \mathbb{Z}$ and $z\in X^{\mathbb{Z}}$. By $z[a,b]$, we mean the sequence $(z(a),z(a+1), \dots, z(b))$. We will refer to $z[a,b]$ as to an \textbf{interval}. Similar notation will be used for open and half open half closed intervals. 
     
     By a \textbf{compact system}, we mean a pair $(X,f)$ where $X$ is a compact metric space and $f$ is a self-homeomorphism of $X$. In this paper, by a system or a dynamical system, we always mean a compact system. A system is called \textbf{minimal} if it does not contain any proper invariant closed subset. In other words, every orbit $\{S^n(x):n\in\mathbb{Z}\}$ denoted by $O(x)$ is dense in $X$, where $x$ is a point in $X$. By a \textbf{subsystem} of $(X,T)$, we mean a closed $T$-invariant subset of $X$. Two systems $(X,f),(Y,g)$ are said to be \textbf{conjugate} (or \textbf{isomorphic}) if there is a homeomorphism $h:X\rightarrow Y$ such that $h \circ f= g\circ h$. We call such $h$ an \textbf{isomorphism} between $(X,f)$ and $(Y,g)$.  It is easy to check that this is an equivalence relation for minimal compact systems. Also, two systems $(X,f)$ and $(Y,g)$ are $\textbf{flip conjugate}$ if either $(X,f)$ and $(Y,g)$ or $(X,f^{-1})$ and $(Y,g)$ are conjugate. A \textbf{factor map} from $(X,f)$ to $(Y,g)$ is a continuous surjection $\pi$ from $X$ to $Y$, such that $\pi \circ f=g\circ \pi$. 

     We say that a dynamical syatem $(Y,S)$ is \textbf{equicontinuous}  if the function class $\{S^n:n\in\mathbb{Z}\}$ is equicontinuous. Every minimal flow $(X,T)$ admits the unique \textbf{maximal equicontinuous factor} system $(Y,S)$ together with a factor map $f$ from $(X,T)$ to $(Y,S)$ such that for any other equicontinuous system $(Z,R)$ together with a factor map $g$ from $(X,T)$ to $(Z,R)$, there is a factor map $h$ from $(Y,S)$ to $(Z,R)$ such that $g=h\circ f$. For more details on this, see \cite{Ellis} and \cite{Williams}.
     
     Let $X$ be a compact metric space. By a \textbf{subshift} we mean a subsystem of $(X^{\mathbb{Z}},S)$ where $S$ denotes the left shift action, i.e. 
    $$
    S(x)(n)=x(n+1),\,\, \forall n\in \mathbb{Z}.
    $$

    To study the complexity of the conjugacy relation of minimal compact systems, we need to encode all minimal compact systems as elements of a standard Borel space. We will consider the minimal subsystems of $(([0,1]^{\mathbb{N}})^{\mathbb{Z}},S)$. 
       Indeed, for every minimal compact system $(X,f)$, we can find an embedding $i$ from $X$ to $[0,1]^{\mathbb{N}}$. Thus, we can embed $(X,f)$ to $(([0,1]^{\mathbb{N}})^{\mathbb{Z}},S)$ by $e$ as:
    $$
    e(x)(n) = i(f^n(x)).
    $$
    Consider the system $(e(X),S))$. This system is obviously conjugate to $(X,f)$. Thus, we can view all minimal compact systems as minimal subsystems of $(([0,1]^{\mathbb{N}})^{\mathbb{Z}},S)$. By considering the Vietoris topology, we can topologize the space of all minimal compact systems. Also, it is well-known that the the set of all minimal subshifts is Borel which provides a standard Borel structure.

    A \textbf{pointed minimal compact system} is a minimal compact system $(X,f)$ together with a point $x\in X$, written as $(X,f,x)$. Two pointed minimal compact systems $(X,f,x)$ and $(Y,g,y)$ are \textbf{conjugate} if there is an isomorphism from $(X,f)$ to $(Y,g)$ sending $x$ to $y$. To topologize the space of all pointed minimal compact systems, we need to consider the subspace topology of the product space  $\mathcal{K}(([0,1]^{\mathbb{N}})^{\mathbb{Z}})\times ([0,1]^{\mathbb{N}})^{\mathbb{Z}}$, where $\mathcal{K}(([0,1]^{\mathbb{N}})^{\mathbb{Z}})$ is the class of compact subsets of $([0,1]^{\mathbb{N}})^{\mathbb{Z}}$ with the Vietoris topology. The set of all minimal compact systems is known to be a Borel subset of $\mathcal{K}(([0,1]^{\mathbb{N}})^{\mathbb{Z}})$ (see \cite[Page 547]{Foreman}). From this it easily follows the set of pointed minimal compact systems is a Borel subset of $\mathcal{K}(([0,1]^{\mathbb{N}})^{\mathbb{Z}})\times ([0,1]^{\mathbb{N}})^{\mathbb{Z}}$. For the convenience of the reader we present a simple folklore proof (which also appears in \cite[Lemma 2.8]{DGKKK}):

    \begin{fact}
        \begin{enumerate}
            \item The set of all minimal subsystems of $(([0,1]^{\mathbb{N}})^{\mathbb{Z}},S)$ is a $G_{\delta}$ subset of $\mathcal{K}(([0,1]^{\mathbb{N}})^{\mathbb{Z}})$.
            \item The set of all pointed minimal subsystems of $(([0,1]^{\mathbb{N}})^{\mathbb{Z}},S)$ is a $G_{\delta}$ subset of $\mathcal{K}(([0,1]^{\mathbb{N}})^{\mathbb{Z}})\times ([0,1]^{\mathbb{N}})^{\mathbb{Z}}$.
        \end{enumerate} 
       
    \end{fact}
        
    \begin{proof}
         (1) Let $(X,S)$ be a subsystem of $(([0,1]^{\mathbb{N}})^{\mathbb{Z}},S)$. Fix a countable basis $\{U_i\}$ of $([0,1]^{\mathbb{N}})^{\mathbb{Z}}$. It is routine to check that $(X,S)$ is subsytem of $(([0,1]^{\mathbb{N}})^{\mathbb{Z}},S)$ if and only if 
         $$
               \forall i\in \mathbb{N}\,(U_i \cap X \neq \emptyset \Leftrightarrow S(U_i) \cap X \neq \emptyset).            $$
               And it is minimal subsystem if and only if
               $$
               \forall i\in \mathbb{N}\,(U_i \cap X \neq \emptyset \Leftrightarrow S(U_i) \cap X \neq \emptyset)\wedge
               \forall i\in \mathbb{N}\,(U_i \cap X \neq \emptyset \Rightarrow \exists\, n\in \mathbb{N}\,\, X\subset \bigcup_{k=0}^n S^k U_i).               
               $$
          which is a $G_{\delta}$ condition. (2) follows easily from (1).
    \end{proof}

       The conjugacy relation of pointed minimal compact system is a Borel equivalence relation \cite{Kaya2}. The conjugacy relation of minimal compact system is not Borel \cite{DGKKK}.
\section{Toeplitz systems and Oxtoby systems}\label{Toepliz}
    We will consider the Toeplitz systems considered by Williams \cite{Williams}. Let $X$ be a compact metric space and $z \in X^{\mathbb{Z}}$ be a sequence. Let $p$ be 
 a natural number, we define ${\rm Per}_{p}(z)=\{n\in \mathbb{Z}|\forall m\equiv n({\rm mod}\,p)\, z(m)=z(n)\}$ and let ${\rm Per}(z)=\bigcup_{n\in \mathbb{N}} {\rm Per}_n(z)$. We call ${\rm Per}(z)$ the \textbf{periodic part} of $z$. Also, we let ${\rm Aper}(z)=\mathbb{Z}\setminus {\rm Per}(z)$.
    \begin{definition}
        Let $X$ be a compact metric space. A sequence $z \in X^{\mathbb{Z}}$ is a \textbf{Toeplitz sequence} if ${\rm Per}(z)=\mathbb{Z}$. A \textbf{Toeplitz system} is the orbit closure (the closure of its orbit in $X^\mathbb{Z}$) of some Toeplitz sequence with the shift map. 
    \end{definition}

    It is well-known that Toeplitz systems are minimal.
    We say $p$ is an \textbf{essential period} of $z$, if ${\rm Per}_p(z) \neq {\rm Per}_q(z) $ for all $q<p$. We will only consider \textbf{non-periodic} Toeplitz sequences, that is, sequences $z$ satisfying
    $\forall n\in \mathbb{Z}, S^nz \neq z$. 
    By a \textbf{period structure} of a non-periodic Toeplitz sequence $z$, we mean a sequence of natural numbers $(p_i)$ such that

    \begin{enumerate}
        \item $p_i$ is an essential period of $z$.
        \item $p_i\mid p_{i+1}$.
        \item  $\bigcup_{i\in \mathbb{N}} {\rm Per}_{p_i}(z)=\mathbb{Z}$.
    \end{enumerate}
   
    Note that every non-periodic Toeplitz sequence has a period structure and the period structure is not uniquely determined. 
    
    %For a Toeplitz sequence $z\in X^{\mathbb{Z}}$, given a interval $[m,n]$, the \textbf{essential period} of $[m,n]$ in $z$ is the least positive integer $p$ such that $z(kp+i)=z(i)$ for all $k\in \mathbb{Z}$, $m\le i\le n$. In such case, we say the interval  is $p$-\textbf{periodic} in $z$.  

    Now let $z$ be a non-periodic Toeplitz sequence and $(p_i)$ be its period structure. We will use the same notation as in \cite{Williams}. Let 
    $$
    A^i_n=\{S^mz|m\equiv n (\mbox{mod}\,p_i)\}
    $$ 
    where $0 \leq n < p_i$. Denote the orbit closure of $z$ by $\overline{O}(z)$. By the \textbf{$p$-skeleton} of $x\in \overline{O}(z)$, we will mean the subsequence of $x$ restricted to ${\rm Per}_{p}(x)$. In other words, when we say $x$ and $y$ have the same $p$-skeleton, we mean ${\rm Per}_{p}(x)={\rm Per}_p(y)$ and $x|_{{\rm Per}_{p}(x)}=y|_{{\rm Per}_{p}(y)}$.
    We will use the following lemma by Williams,
    
    \begin{lemma}\label{W1}
        \rm (Williams), \cite[Lemma 2.3]{Williams}. Let $X$ be a compact metric space, $z$ be a non-periodic Toeplitz sequence in $X^{\mathbb{Z}}$ and $(p_i)$ be the period structure of $(\overline{O}(z),S)$. Then for all $i\in \mathbb{N}$, $\{\overline{A^i_n}\},n=0,1,\dots,p_i-1$ is a partition of $\overline{O}(z)$ into clopen sets. Also, $A^i_n\subset A^j_m$ for $j<i,,m\equiv n$(mod$\,p_j$) and $SA^i_n=A^i_{n+1}$.
    \end{lemma} 

    Fix a sequence of natural numbers $(p_i)$ such that $p_i\mid p_{i+1}$. Let $G=\varprojlim \mathbb{Z}_{p_i}$. In other words,
    $$
    G=\{(n_i)\in\prod_i \mathbb{Z}_{p_i}|\forall j<i,\, n_i \equiv n_j ({\rm mod}\,p_j)\},
    $$
    topologized by the subspace topology of the product space $\prod_i \mathbb{Z}_{p_i}$.
    Denote the element $(1,1,\dots)$ by $\Bar{1}$. Note that $G$ is a subgroup of $\prod_i \mathbb{Z}_{p_i}$, thus there is a natural automorphism $\widehat{1}$ on $G$:
    $$
    \widehat{1}(g)=g+\Bar{1}.
    $$

     Suppose $(\overline{O}(z),S)$ is a Toeplitz system with perodic structure $(p_i)$. Then $(G,\widehat{1})$ is the maximal equicontinuous factor of $(\overline{O}(z),S)$ by \cite[Theorem 2.2]{Williams}. Note that $G$ is determined by the period structure of $z$.
    Let $\pi :(\overline{O}(z),S) \rightarrow (G,\widehat{1})$ be defined so that
    $\pi^{-1}((n_i))= \bigcap_i \overline{A^i_{n_i}}$. By Lemma \ref{W1}, this is well defined. Note that $\pi$ is a factor map. We will call this $\pi$ the \textbf{canonical factor map} of $(\overline{O}(z),S)$. 

    Let $X$ be a compact metric space. Let $z_1$ and $z_2$ be two non-periodic Toeplitz sequences in $X^{\mathbb{Z}}$ with the same period structure. Suppose there is an isomorphism $f$ between $(\overline{O}(z_1),S)$ and $(\overline{O}(z_2),S)$. Let $(G_1,\widehat{1})$ and $(G_2,\widehat{1})$ be the maximal equicontinuous factors of $(\overline{O}(z_1),S)$ and $(\overline{O}(z_2),S)$. Let $\pi_1$ and $\pi_2$ be the canonical factor maps of $(\overline{O}(z_1),S)$ and $(\overline{O}(z_2),S)$ respectively. Since they have the same period structure, we know that $G_1=G_2$. By the definition of the maximal equicontinuous factor, there will be an isomorphism $h:G_1\rightarrow G_2$ such that $h\circ \pi_1=\pi_2\circ f$.

    \begin{lemma}\label{EQF}
        Let $X$ be a compact metric space and $z$ be a non-Topelitz sequence in $X^{\mathbb{Z}}$. Let $G$ be the maximal equicontinuous factor of $(\overline{O}(z),S)$ and $\pi: (\overline{O}(z),S) \rightarrow (G,\widehat{1})$  be the canonical factor map. Then 
       \begin{itemize}
           \item[(i)] for any $g\in G$, every element in $\pi^{-1}(g)$ has the same periodic and non-periodic part,
           \item[(ii)] if $y$ is a Toeplitz sequence in $\overline{O}(z)$, then $\pi^{-1}(\pi(y))=\{y\}$.
       \end{itemize}
    \end{lemma}
    \begin{proof}
        (i) Suppose $g=(n_i)$ and $\pi^{-1}(g)= \bigcap_i \overline{A^i_{n_i}}$. By the definition of $A^i_n$, all elements in $\overline{A^i_n}$ have the same $p_i$-period structure as $S^nz$ and this holds for all $i$.

        (ii) Suppose $y\in \overline{O}(z)$ is a Toeplitz sequence. Then all positions of $y$ are periodic. Therefore, only the element $y$ itself can have the same $p_i$-skeleton as $y$ for all $i$.
    \end{proof}

     \noindent \textbf{Notation}. Given a Topelitz system $(\overline{O}(z),S)$ and its maximal equicontinuous factor $(G,\widehat{1})$, for $g\in G$, we will use notations ${\rm Per}(g)$ and ${\rm Aper}(g)$ to denote the periodic part and non-periodic part of any elements in $\pi^{-1}(g)$. This is well-defined by Lemma \ref{EQF}.

     \vspace{0.5em}
     The following lemma is well-known but we prove it for completeness.
     \begin{lemma}\label{non-Toeplitz exist}
         Let $X$ be a compact metric space. Any infinite Topelitz system $(Y,S)$ in $(X^{\mathbb{Z}},S)$ contains a non-Toeplitz sequence.
     \end{lemma}
     \begin{proof}
         Take a periodic structure $(p_i)$ of $(Y,S)$ and a Topelitz sequence $y\in Y$. Since $Y$ is infinite, for any $i$, there exists infinitely many natural numbers $n$ such that $n\not \in {\rm Per}_{p_i}(y)$. Let $(n_i)$ be an increasing sequence of natural numbers such that $n_i\not \in {\rm Per}_{p_i}(y)$, by compactness and taking a subsequence, we may assume $(S^{n_i}y)$  converges to $y'$. Since $0\not \in {\rm Per}_{p_i}(S^{n_j}y)$ for all $j\geq i$, we have $0\not \in {\rm Per}_{p_i}(y')$ for all $i$. Thus, $0\not \in {\rm Per}(y')$.
     \end{proof}
     Now let us introduce \textbf{Oxtoby systems}. Subshifts generated by Oxtoby sequences were also studied in \cite{Williams} where Williams proved every compact metric space could be realized as the set of ergodic measures of such a minimal compact system. To define such systems, first we let $X$ be a compact metric space, $(x_i)$ be a countable subset of $X$.   
    
    \vspace{0.5em}
    
    \begin{definition}
        Let $(p_i)$ be a sequence of natural numbers such that:

          (0) $p_0=1$.
          
          (1) $p_i | p_{i+1}$,  $ \forall i\in \mathbb{N}$.
          
          (2) $p_1\geq 3$ and $\forall i$, $\frac{p_{i+1}}{p_i} \geq 3$. 

          We call such sequences of $(p_n)$ a \textbf{fast growing} sequence.
    \end{definition}
          
          A sequence $z\in X^{\mathbb{Z}}$ is called an \textbf{Oxtoby sequence} (generated by a fast growing sequence $(p_i)$ and $(x_i)$) if it is defined by the the following inductive construction: In the first step, we set $z(n)=x_1$ for all $n\equiv -1$ or $0$ mod $p_1$. For $i\in \mathbb{N}$ and $k\in \mathbb{Z}$ we let $J(i,k)$ denote the set of $n\in [kp_i, (k+1)p_i)$ for which $z(n)$ has not been defined at the end of the $i^{\rm th}$ step. The $(i+1)^{\rm th}$ step is to set $z(n)=x_{i+1}$ for $n\in J(i,k)$ with $k\equiv -1$ or $0$ mod $\frac{p_{i+1}}{p_i}$. By (1) and (2), $z$ is well defined and Toeplitz. 
       
           Denote such a sequence by $z((p_i),(x_i))$ and denote its orbit closure by $\overline{O}(z((p_i),(x_i)))$.

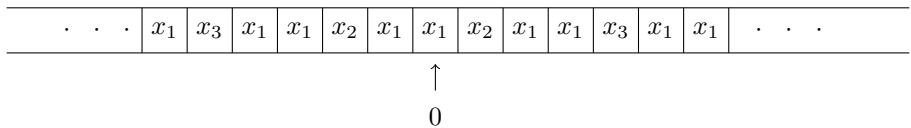
\begin{figure}[h!]
\centering
\begin{tikzpicture}
\usetikzlibrary{math}

\usetikzlibrary{decorations.pathreplacing}

\node (x0) at (-3 * 0.4, 0.0)    {$\cdot$} ;
\node (x0) at (-2 * 0.4, 0.0)    {$\cdot$} ;
\node (x0) at (-1 * 0.4, 0.0)    {$\cdot$} ;
\node (00) at ( 0 * 0.4, 0.0)    {} ;
\node (01) at ( 1 * 0.4, 0.0)    {} ;
\node (02) at ( 2 * 0.4, 0.0)    {} ;
\node (03) at ( 0.1, 0.0)    {$x_1$} ;
\node (04) at ( 0.7, 0.0)    {$x_3$} ;
\node (05) at ( 1.3, 0.0)    {$x_1$} ;
\node (06) at ( 1.9, 0.0)    {$x_1$} ;
\node (07) at ( 2.5, 0.0)    {$x_2$} ;
\node (08) at ( 3.1, 0.0)    {$x_1$} ;
\node (09) at ( 3.7, 0.0)    {$x_1$} ;
\node (10) at ( 4.3, 0.0)    {$x_2$} ;
\node (11) at ( 4.9, 0.0)    {$x_1$} ;
\node (12) at ( 5.5, 0.0)    {$x_1$} ;
\node (13) at ( 6.1, 0.0)    {$x_3$} ;
\node (14) at ( 6.7, 0.0)    {$x_1$} ;
\node (15) at ( 7.3, 0.0)    {$x_1$} ;
\node (16) at (16 * 0.4, 0.0)    {} ;
\node (17) at (17 * 0.4, 0.0)    {} ;
\node (18) at (18 * 0.4, 0.0)    {} ;
\node (19) at (19 * 0.4, 0.0)    {} ;
\node (20) at (20 * 0.4, 0.0)    {$\cdot$} ;
\node (20) at (21 * 0.4, 0.0)    {$\cdot$} ;
\node (20) at (22 * 0.4, 0.0)    {$\cdot$} ;
\node      at (3.7, -1.15) {0}       ;
\draw[-] (-5 * 0.4, 0.3) -- (25 * 0.4, 0.3) ;
\draw[-] (-5 * 0.4, -0.3) -- (25 * 0.4, -0.3) ;
\draw[-] (2.2, -0.3) -- (2.2, 0.3) ;
\draw[-] (3.4, -0.3) -- (3.4, 0.3) ;
\draw[-] (5.8, -0.3) -- (5.8, 0.3) ;
\draw[-] (4.6, -0.3) -- (4.6, 0.3) ;
\draw[-] (7, -0.3) -- (7, 0.3) ;
\draw[-] (1, -0.3) -- (1, 0.3) ;
\draw[-] (-0.2, -0.3) -- (-0.2, 0.3) ;
\draw[-] (2.8, -0.3) -- (2.8, 0.3) ;
\draw[-] (4, -0.3) -- (4, 0.3) ;
\draw[-] (5.2, -0.3) -- (5.2, 0.3) ;
\draw[-] (6.4, -0.3) -- (6.4, 0.3) ;
\draw[-] (0.4, -0.3) -- (0.4, 0.3) ;
\draw[-] (1.6, -0.3) -- (1.6, 0.3) ;
\draw[-] (7.6, -0.3) -- (7.6, 0.3) ;
\draw[->] (3.7, -0.3-0.5) -- (3.7, 0.3-0.75) ;

\end{tikzpicture}
\caption{An Oxtoby sequence with $p_n=3^n$}
\end{figure}

           \vspace{0.5em}

       Note that every Oxtoby sequence is a Toeplitz sequence. Oxtoby systems have a lot of interesting properties, for example, all aperiodic positions of a non-Toeplitz sequence in such a system are occupied by the same element in $X$.  More precisely, we have the following:
        \vspace{0.5em}
        
        \begin{lemma} \label{CS}
            {\rm (Williams, \cite[Lemma 3.3]{Williams})}  Let $X$ be a compact metric space, $(x_i)$ be a sequence of elements in $X$ and $(p_i)$ be a fast growing sequence. Let $z((p_i),(x_i))$ be an Oxtoby sequence and $y\in \overline{O}(z((p_i),(x_i)))$, then for all $n,m\in {\rm Aper}(y)$, we have $y(n)=y(m)$.
        \end{lemma}

       We need the following lemma to study isomorphism type of Oxtoby systems which will be used in Definition $\ref{DF1}$.

       \begin{lemma}\label{OL1}
           Let $X$ be a compact metric space, $(x_i)$ be a sequence of elements in $X$ and $(p_i)$ be a fast growing sequence, $z((p_i),(x_i))$ is an Oxtoby sequence. Let $j\geq i\in \mathbb{N}$ and $k,k'\in \mathbb{Z}$. If $[kp_i,(k+1)p_i)\subset [k'p_j,(k'+1)p_j)$, then either $[kp_i,(k+1)p_i)\subset {\rm Per}_{p_j}(z((p_i),(x_i)))$ or $J(j,k')\cap [kp_i,(k+1)p_i)=J(i,k)$.
       \end{lemma}
        \begin{proof}
              Note that given $i\leq j$, since $p_i\mid p_j$ , for every $k$, there is a unique $k'$ such that $[kp_i,(k+1)p_i)\subset [k'p_j,(k'+1)p_j)$ holds. We prove the statement of the lemma by induction on $j$. We start with $j=i$. Note that our assumption $[kp_i,(k+1)p_i)\subset [k'p_j,(k'+1)p_j)$, implies that $k=k'$, thus we are done. Now we are proceed to the induction step. Suppose the statement of the lemma is true for $j\geq i$. In the $(j+1)^{\rm th}$ step, note that since $p_i\mid p_j$ and $p_i\mid p_{j+1}$, an interval $[kp_i,(k+1)p_i)$ is either contained in $[-p_{j}+mp_{j+1},p_{j}+mp_{j+1})$ for some $m\in \mathbb{Z}$  or $z[kp_i,(k+1)p_i)$ remains the same as what it was in the $j^{\rm th}$ step. In the first case,  we get $[kp_i,(k+1)p_i)\subset {\rm Per}_{p_{j+1}}(z((p_i),(x_i)))$ by the definition of the Oxtoby sequence. In the second case, we get either $[kp_i,(k+1)p_i)\subset {\rm Per}_{p_j}(z((p_i),(x_i)))$ or $J(j,k')\cap [kp_i,(k+1)p_i)=J(i,k)$ by the induction hypothesis.
        \end{proof}
        
        \section{Topological type of sequences}\label{TT relation}
       
         The following definition is classical.
        \begin{definition}
            The \textbf{$\omega$-limit set} of a sequence $(x_i)$ of elements of a compact metric space $X$, denoted by $L((x_n))$, is the set of limits of sequences $(x_{n_k})$ where $(n_k)$ is a subsequence of $\mathbb{N}$. In other words, 
            $$
            L((x_n))=\{ x\in  X | \exists\,\,{\rm an\,\,increasing}\,\, (n_k)\,\,x={\rm lim}_k x_{n_k}\}.
            $$
        \end{definition}
         The following definition is new.
        \begin{definition}
            Two sequences $(x_n)$ and $(x_n')$ are said to have the same \textbf{topological type} if for every increasing sequence of natural numbers $(n_k)$,
         $$
         (x_{n_k})\,\,{\rm converges}\,\,{\rm iff}\,\,(x'_{n_k})\,\,{\rm converges}.
         $$
         \end{definition}
        This can be regraded as an equivalence relation on the space $X^{\omega}$. We denote the topological type equivalence relation by $E_{{\rm tt}}(X)$. From the definition, this is a $\Pi^1_1$ equivalence relation, actually it is also Borel which will be proved in Section \ref{Borel reduction} (Theorem \ref{BR}). In Section \ref{Borel reduction}(Theorem \ref{R}), we will also prove that $E_{{\rm tt}}([0,1]^{\mathbb{N}})$ is  Borel reducible to the isomorphism relation of minimal compact systems. The following facts are routine to check but important.

        \begin{fact}\label{TT implies homeo}
            Let $(x_i)$ and $(y_i)$ be two sequences in a compact metric space $X$. If $(x_i)E_{\rm tt}(X)(y_i)$, then the function $f: L((x_i))\rightarrow L((y_i))$ such that for any increasing sequence of natural numbers $(n_k)$
            $$
            {\rm lim}_k x_{n_k}\mapsto {\rm lim}_ky_{n_k}
            $$
             is well defined and is a homeomorphism.
        \end{fact}
         \begin{proof}
             It is routine to check.
         \end{proof}
        
        It is obvious from the definition that $\omega$-limit set of any sequence in a compact metric space is nonempty and closed and two sequences with the same topological type have homeomorphic $\omega$-limit sets.

        The following lemma shows the connection of topological type of a sequence and isomorphism between minimal compact systems.

        \begin{lemma}\label{min}
            Let $(X,f,x)$ and $(Y,g,y)$ be two pointed minimal compact systems. The following are equivalent:
            
            \begin{enumerate}
                \item  $(X,f,x),(Y,g,y)$ are conjugate.
                \item  For any sequence of integers $(n_k)$, $(f^{n_k}(x))$ is convergent if and only $(g^{n_k}(y))$ is convergent.
                \item  $(f^n(x))_{n\in \mathbb{N}}$ and $(g^n(y))_{n\in \mathbb{N}}$ have the same topological type and there exists an increasing sequence $(m_k)$ of natural numbers such that $f^{m_k}(x)\rightarrow x$ and $g^{m_k}(y)\rightarrow y$.
                
            \end{enumerate} 
        \end{lemma} 
        \begin{proof}
           % \color{ref}(1)$\Rightarrow$(2). If the two pointed minimal compact systems are conjugate by %$h$, we have $(f^n(x))$ and $(g^n(h(x)))$ have the same topological type. %Also, since any orbit is dense, there exists an increasing sequence %$(n_k)$ such that $f^{n_k}(x)\rightarrow x$. Since $h$ is an isomorphism, %we have 
            %$$
            %h(x)={\rm lim}_k h(f^{m_k}((x)))= {\rm lim}_k g^{m_k}(h(x)).
            %$$  

            (1)$\Rightarrow$(2). Let $h$ be the conjugacy map between $(X,f,x)$ and $(Y,g,y)$. Thus for any convergent sequence $(f^{n_k}x)$, we have
            $$
             {\rm lim}_k g^{n_k}y={\rm lim}_kg^{n_k}h(x)={\rm lim}_k h(f^n_k(x)).
            $$
            which exists since $h$ is continuous. 

            (2)$\Rightarrow$(3). By definition, we have the sequence $(f^n(x))$ and $(g^n(y))$ have the same topological type. By minimality, we can find an increasing subsequence $(m_k)$ such that $f^{m_k}x$ converges to $x$. Now take $(n_k)$ to be the sequence such that $n_{2k}=0$ and $n_{2k+1}=m_k$, we have $f^{n_k}(x)$ converges to $x$. Thus $g^{n_k}(y)$ also converges, since $g^{n_{2k}}y=y$, we know $g^{n_k}(y)$ converges to $y$, in particular, $g^{n_{2k+1}}(y)=f^{m_k}(y)$ converges to $y$.

            (3)$\Rightarrow$(1). Since $L((f^n(x)))$ is a nonempty invariant closed subset of $X$, we have $X=L((f^n(x)))$. Then any $x_0\in X$ can be written as the limit of $f^{m_k}(x)$ for some increasing $(n_k)$. Define $h(x_0)$ to be the limit of $g^{m_k}(y)$, by Fact \ref{TT implies homeo}, this is a well-defined isomorphism and this isomorphism sends $x$ to $y$.
        \end{proof}
        \begin{Corollary}\label{BRBR}
            The following are equivalent:
            \begin{enumerate}
                \item  $(X,f,x),(Y,g,y)$ are conjugate.
                \item  The following two sequences have the same topological type:
                $$
                 (x,f(x),x,f^2(x),x,f^3(x),x,\dots)\,\,\,\,\mbox{and}\,\,\,\,(y,g(y),y,g^2(y),y,g^3(y),y,\dots)
                $$
            \end{enumerate}
        \end{Corollary}
        \begin{proof}
             (1)$\Rightarrow$(2). Take $(n_k)$ to be a subsequence of $(0,1,0,2,0,3,\dots)$, then this follows by Lemma \ref{min}(2).
             
            (2)$\Rightarrow$(1). This follows by Lemma \ref{min}(3).
           
        \end{proof}
              The condition (3) in Lemma \ref{min} can not be weakened to say that $(f^n(x))$ and $(g^n(y))$ have the same topological type. Indeed, in \cite[Corollary 7.4]{DDMP}, it is shown that every rank 2 minimal symbolic subshift $(X,S)$ has a unique asymptotic pair. This implies that there exists a pair of distinct points $x,y$ in $X$ such that for any increasing sequence $(n_k)$ of natural numbers, $(S^{n_k}x)$ converges if and only if $(S^{n_k}y)$ converges and if they converge then
              \begin{equation}\label{R2M}
                  {\rm lim}_kS^{n_k}(x)={\rm lim}_kS^{n_k}(y).
              \end{equation}
              In particular, $(S^nx)$ and $(S^ny)$ have the same topological type. However, there is no automorphism sending $x$ to $y$. Indeed, suppose there exists such an automorphism $f$. Since in a minimal system, the forward orbit of every point is dense, there is an increasing sequence $(n_k)$ of natural numbers such that $S^{n_k}x$ converges to $x$, but by (\ref{R2M}), this means 
              $$
              f(x)={\rm lim}_kf(S^{n_k}x)={\rm lim}_kS^{n_k}f(x)= {\rm lim}_kS^{n_k}(y)={\rm lim}_kS^{n_k}(x)=x,
              $$ 
              a contradiction.

                  The conjugacy of minimal compact system is a little different than the conjugacy of pointed minimal compact systems. First we recall the following standard fact.

               \begin{fact}\label{strong TT implies homeo}
             Let $(x_i)$ and $(y_i)$ be two sequences in a compact metric space $X$. Then the following are equivalent:
             \begin{enumerate}
                 \item For any sequence of natural numbers $(n_k)$, we have $(x_{n_k})$ is convergent if and only if $(y_{n_k})$ is convergent.
                 \item There is a homeomorphism from $\overline{\{(x_i)\}}$ to $\overline{\{(y_i)\}}$ sending $x_i$ to $y_i$.
             \end{enumerate}
         \end{fact}
         
         \begin{proof}
             It is routine to check.
         \end{proof}

              Now we have the following version of Lemma \ref{min}.
              
              \begin{lemma}\label{min2}
            Let $(X,f)$ and $(Y,g)$ be two minimal compact systems. The following are equivalent:
            
            \begin{enumerate}
                \item $(X,f),(Y,g)$ are conjugate.
                 \item There exists $x\in X$ and $y\in Y$ such that for all sequences of integers $(n_k)$, we have that $(f^{n_k}(x))$ converges if and only if $(g^{n_k}(x))$ converges.
                \item There exists $x\in X$ and $y\in Y$ such that $(f^n(x))_{n\in \mathbb{N}}$ and $(g^n(y))_{n\in \mathbb{N}}$ have the same topological type.
               \end{enumerate}
         \end{lemma}
         \begin{proof}
             (1)$\Rightarrow$(2) Directly by Lemma \ref{min}. 
             
             (2)$\Rightarrow$(3) is trivial.
             
             (3)$\Rightarrow$(1)  Since $L((f^n(x)))$ is a nonempty invariant closed subset of $X$ and $X$ is minimal, we have $X=L((f^n(x)))$. Then any $x_0\in X$ can be written as the limit of $f^{m_k}(x)$ for some increasing $(m_k)$. Define $h(x_0)$ to be the limit of $g^{m_k}(y)$, then by Fact \ref{strong TT implies homeo}, this is a well defined conjugacy.
         \end{proof}
         \begin{lemma}\label{Will1}
             Let $X$ be a compact metric space, $(x_i)$ be a sequence of elements of $X$ and $(p_i)$ be a fast growing sequence. Let $z((p_i),(x_i))$ be an Oxtoby sequence. Let $\pi$ be the canonical factor map of $(\overline{O}(z((p_i),(x_i))),S)$. Assume $x\in (\overline{O}(z((p_i),(x_i))))$ is such that $\pi(x)=(n_i)$ and ${\rm Aper}(x)\neq \emptyset$. Then 
             \begin{enumerate}
                 \item Both $(n_i)$ and $(p_i-n_i)$ go to infinity.
                 \item For any $k\in {\rm Aper}(x)$ and $i\in \mathbb{N}$, if $-n_i<k<p_i-n_i$, then $S^{n_i}z(k)=x_{i+1}$.
                 
             \end{enumerate}
         \end{lemma}
          \begin{proof}
            (1) First note that by the definition of canonical factor map, the sequence $(n_i)$ is increasing. And since $n_{i}\equiv n_{i+1}$ mod $p_i$ and $p_i\mid p_{i+1}$, thus, $p_{i+1}-n_{i+1}\geq p_{i+1}-(p_{i+1}-p_i)-n_i=p_i-n_i$. In conclusion, both sequences $(n_i)$ and $(p_i-n_i)$ are non-decreasing. Now suppose $(n_i)$ is eventually constant equal to $n$, then $y=S^nz$ which is Toeplitz, a contradiction. Finally suppose $(p_i-n_i)$ is eventually constant, say equal to $n$, then $y=S^{-n}z$ which is also Toeplitz, a contradiction. 

            (2) Let $g=(n_i)$ and $k\in {\rm Aper}(g)$. Write $z$ for $z((p_i),(x_i))$. Recall the notation $J(i,k)$ from the definition of an Oxtoby sequence and note that for all $m\in J(i,0)$ we have $z(m)=x_{i+1}$, and thus $S^{n_i}z(m-n_i)=x_{i+1}$ for all $m\in J(i,0)$. Note that ${\rm Per}_{p_i}(g)={\rm Per}_{p_i}(S^{n_i}z)$, so, in particular, ${\rm Per}_{p_i}(g)\cap [-n_i,p_i-n_i)={\rm Per}_{p_i}(S^{n_i}z)\cap [-n_i,p_i-n_i)$. On the other hand, by the definition of $J(i,0)$, the set ${\rm Per}_{p_i}(z)$ is complementary to
            $J(i,0)$ on $[0,p_i)$, so ${\rm Per}_{p_i}(S^{n_i}z)$ is  complementary to $\{m-n_i|m\in J(i,0)\}$ on $[-n_i,p_i-n_i)$. Thus, ${\rm Per}_{p_i}(g)$ is complementary to $\{m-n_i|m\in J(i,0)\}$ on $[-n_i,p_i-n_i)$. This implies
            \begin{equation}\label{AP111}
                {\rm Aper}(g)\cap [-n_i,p_i-n_i)\subseteq \{m-n_i|m\in J(i,0)\}.
            \end{equation}
             By (\ref{AP111}), we know that for $k\in {\rm Aper}(x)$ and $-n_i<k<p_i-n_i$, we can find $m\in J(i,0)$ such that $k=m-n_i$, thus 
             $$
              S^{n_i}z(k)=S^{n_i}z(m-n_i)=x_{i+1}.
             $$
            
        \end{proof}
        \begin{lemma}\label{Will}
             Let $X$ be a compact metric space, $(x_i)$ be a sequence of elements of $X$ and $(p_i)$ be a fast growing sequence. Let $z((p_i),(x_i))$ be an Oxtoby sequence. Let $\pi$ be the canonical factor map of $(\overline{O}(z((p_i),(x_i))),S)$. Assume $g\in \pi(\overline{O}(z((p_i),(x_i))))$ is such that ${\rm Aper}(g)\ne\emptyset$. Then we have:
             \begin{enumerate}
                 \item For all $k\in {\rm Aper}(g)$ and $y\in \pi^{-1}(g)$ we have $y(k)\in L((x_i))$.
                 \item For all $x \in L((x_i))$, we can find $y\in \pi^{-1}(g)$, such that for all $k\in {\rm Aper}(g)$, we have $y(k)=x$.
             \end{enumerate}
        \end{lemma}
       
        \begin{proof}
            %Let $g=(n_i)$, $x\in L((x_i))$ and $k\in {\rm Aper}(g)$. Write $z$ for $z((p_i),(x_i))$. Recall the notation $J(i,k)$ from the definition of an Oxtoby sequence and note that for all $m\in J(i,0)$ we have $z(m)=x_{i+1}$, and thus $S^{n_i}z(m-n_i)=x_{i+1}$ for all $m\in J(i,0)$. Note that ${\rm Per}_{p_i}(g)={\rm Per}_{p_i}(S^{n_i}z)$, so, in particular, ${\rm Per}_{p_i}(g)\cap [-n_i,p_i-n_i)={\rm Per}_{p_i}(S^{n_i}z)\cap [-n_i,p_i-n_i)$. On the other hand, ${\rm Per}_{p_i}(z)$ is disjoint from 
            %$J(i,0)$, so ${\rm Per}_{p_i}(S^{n_i}z)$ is %disjoint from $\{m-n_i|m\in J(i,0)\}$. Thus, ${\rm Per}_{p_i}(g)$ is disjoint from $\{m-n_i|m\in J(i,0)\}$, or, in other words, 
            %\begin{equation}\label{AP}
                %{\rm Aper}(g)\cap [-n_i,p_i-n_i)\subseteq \{m-n_i|m\in J(i,0)\}.
            %\end{equation}

            (1) Since ${\rm Aper}(g)\neq \emptyset$, every element in $\pi^{-1}(g)$ has nonempty aperiodic part. Thus, $\pi^{-1}(g)$ must contain a non-Toeplitz sequence. %Now we argue that both $(n_i)$ and $(p_i-n_i)$ go to infinity. First, suppose $(n_i)$ is eventually constant equal to $n$, then $\pi^{-1}(g)=\{S^nz\}$ which is Toeplitz, a contradiction. Now suppose $(p_i-n_i)$ does not converge to infinity. Since $n_i<p_i$, we have that $(p_i-n_i)$ is eventually constant, say equal to $n$, then $\pi^{-1}(g)=\{S^{-n}z\}$ which is also Toeplitz, a contradiction. 
            
            By Lemma \ref{Will1}(1), both $(n_i)$ and $(p_i-n_i)$ go to infinity. Thus, eventually there will be some $j_0$ such that for all $j\geq j_0$ we have $-n_j<k<p_j-n_j$. By Lemma \ref{Will1}(2) and the construction of an Oxtoby sequence we have 
            $S^{n_j}z(k)=x_{j+1}$
            for all $j\geq j_0$. For any $y\in \pi^{-1}(g)$, note that $y$ is the limit of a subsequence of $S^{n_i}z$, hence $y(k)$ is a limit of a subsequence of $(x_i)$, thus we have $y(k)\in L((x_i))$. 
            
            (2) For $x\in L((x_i))$, we can take a subsequence $(x_{i_j})$ of $(x_i)$ which converges to $x$ as $j$ goes to infinity. By Lemma \ref{Will1}(2), we have that $S^{n_{i_j-1}}z(n)=x_{i_j}$ converges to $x$. Denote the limit of $S^{n_{i_j}-1}z$ by $y$. Then $\pi(y)=g$ and $y$ is the desired sequence.
        \end{proof}

        \begin{proposition}\label{SIN}
            Let $X$ be a compact metric space, $(x_i) \in X^{\omega}$. Let $(\overline{O}(z((p_i),(x_i))),S)$  be an Oxtoby system with fast growing sequences $(p_i)$. Let $\pi$ be the canonical factor map of $(\overline{O}(z((p_i),(x_i))),S)$. Then $L((x_i))$ is a singleton if and only if then $\pi$ is an isomorphism.
        \end{proposition}

        \begin{proof}
            Suppose $L(x_i)$ is a singleton. We will show that $\pi$ is injective. By Lemma \ref{EQF}(ii), $\pi$ is injective on Topelitz sequences. Let $x$ be the unique element in $L((x_i))$.  If $y$ is a non-Toeplitz sequence in $\overline{O}(z((p_i),(x_i)))$, then by Lemma \ref{Will}, for every $n\in {\rm Aper}(y)$, we have $y(n)=x$.  Since ${\rm Per}(y)$ and the restriction of $y$ on ${\rm Per}(y)$ is uniquely determined by $\pi(y)$, this shows $\pi^{-1}(\pi(y))$ is a singleton. 

            On the other hand, suppose $\pi$ is injective. Use Lemma \ref{non-Toeplitz exist} to find a non-Toeplitz sequence $y\in \overline{O}(z((p_i),(x_i))$, and write $\pi(y)=g$. Note that ${\rm Aper}(g)\neq \emptyset$ since $y$ is non-Toeplitz. By Lemma \ref{Will}(2), we have that  $|L(x_i)|\leq |\pi^{-1}(g)|=1$ which finishes the proof since $L(x_i)$ is non empty. 
        \end{proof}
          \vspace{0.5em}

            %\noindent \textit{Remarks}: Note that $E_{\rm tt}([0,1]^\mathbb{N})$ is bireducible to the equivalence $E_{\rm csc}$ defined in \cite{DG}. On one hand, we have a canonical way to define an embedding $f$ from the compact metric space $\overline{X_r}$ in \cite{DG} to $[0,1]^\mathbb{N}$, so the reduction from $E_{\rm csc}$ to $E_{\rm tt}([0,1]^\mathbb{N})$ is easy to get by taking an enumeration of $\{f(i)\}$ so that every item occurs infinitely many times. On the other hand, our next section shows the other direction.

    \section{Isomorphism type of Oxtoby systems}\label{proof 1}
        In this section, we study the isomorphism type of Oxtoby systems. The automorphism group of Oxtoby system was studied in \cite{BK} where Bu\l atek and Kwiatkowski proved that any self-conjugacy map of a symbolic Oxtoby system is a shift. In the following theorem, we generalize it as follows. 
        \begin{theorem} \label{Main1}
             Let $(X,d)$ be a compact metric space, $(x_i)$ and $(x'_i)$ are in $X^{\omega}$. Let $(\overline{O}(z((p_i),(x_i))),S)$ and
             $(\overline{O}(z((p_i),(x'_i))),S)$ be two Oxtoby systems with the same fast growing sequences $(p_i)$. Suppose both $L((x_i))$ and $L((x_i'))$ have at least two points. Then for any isomorphism $f$ between  $(\overline{O}(z((p_i),(x_i))),S)$ and $(\overline{O}(z((p_i),(x'_i))),S)$, we have $f(z((p_i),(x_i)))=S^mz((p_i),(x'_i))$ for some $m\in \mathbb{Z}$.
      \end{theorem}

        In Theorem \ref{Main1}, we need the assumption that $L((x_i))$ and $L((x_i'))$ have at least two points otherwise the result is not true.  Indeed, if $L((x_i))$ is a singleton then $(\overline{O}(z((p_i),(x_i))),S)$ is conjugate to its maximal equicontinuous factor by Theorem \ref{SIN}. The maximal equicontinuous factor of any Toeplitz system in particular of  $(\overline{O}(z((p_i),(x_i))),S)$ is an Odometer and in the Odometer any point can be mapped to another point by isomorphism.
        
        To prove this result, we need the following lemmas. Let us recall the notation we are using. If $X$ is a set, $I=[a,b]$ is an interval with $a,b\in \mathbb{Z}$ and $z\in X^{\mathbb{Z}}$, by $z[a,b]$, we mean the sequence $(z(a),z(a+1), \dots, z(b))$. Similar notation is used for open and half open half closed intervals.
        \begin{lemma} \label{L1}
            Let $(X,d)$ be a compact metric space. Let $f$ be an isomorphism between two subshifts $A$ and $B$ in $X^{\mathbb{Z}}$. Then for all $\epsilon>0$, there exists an $n\in \mathbb{N}$, such that for all $x\in A$ and every $m,m'\in \mathbb{Z}$, if $x[m-n,m+n)=x[m'-n,m'+n)$ then $d(f(x)(m),f(x)(m'))<\epsilon$.
        \end{lemma}
        \begin{proof}
            Let $d_1$ be the product metric on $X^{\mathbb{Z}}$. Fix $\epsilon>0$. Since $f$ is uniformly continuous, we can find an $\epsilon_0$ such that 
            $$
            \forall x,y\in X^{\mathbb{Z}},\, d_1(x,y)<\epsilon_0 \Rightarrow d(f(x)(0),f(y)(0))<\epsilon.
            $$
            
            Now find $n$ such that for any $x,y\in X^{\mathbb{Z}}$, 
            $$
            x(-n,n]=y(-n,n] \Rightarrow d_1(x,y)<\epsilon_0.
            $$

            Now fix $x\in A$ and $m,m'\in \mathbb{Z}$ such that $x[m-n,m+n)=x[m'-n,m'+n)$. The latter implies that $d_1(S^mx,S^{m'}x)<\epsilon_0$, which in turn implies $d(f(S^mx)(0),f(S^{m'}x(0)))<\epsilon$. Since $f$ is an isomorphism, $f(S^mx)=S^mf(x)$ and $f(S^{m'}x)=S^{m'}f(x)$. Thus, we have 
            $$
            d(f(x)(m),f(x)(m'))=d(f(S^mx)(0), d(f(S^{m'}x)(0))<\epsilon.
            $$ 
           \end{proof}
          
        \begin{lemma} \label{L2}
            Let $(X,d)$ be a compact metric space. Let $f$ be an isomorphism between two Toeplitz subsystems $(\overline{O}(z),S)$ and $(\overline{O}(z'),S)$ in $(X^{\mathbb{Z}},S)$ where $z$ and $z'$ are Topelitz sequences. Let $(p_i)$ be a period structure of $z$. Then for all $\epsilon>0$, there is $i_0\in\mathbb{N}$ such that for all $i\ge i_0$ and $k,k'\in\mathbb{Z}$, $z[kp_i,(k+1)p_i)=z[k'p_i,(k'+1)p_i)$ implies that for all $0\le s<p_i$, $d(f(z)(kp_i+s),f(z)(k'p_i+s))<\epsilon$.
        \end{lemma}
        \begin{proof}
            Take the $n$ in Lemma \ref{L1}, we look at  $z[-n,n)$, since $z$ is Toeplitz, there exists $i_0\in \mathbb{N}$ such that $[-n,n)\subset {\rm Per}_{p_{i_0}}(z)$. This means for any $i\geq i_0$ and $k\in \mathbb{Z}$, $z[kp_i-n,kp_i+n)=z[-n,n)$. Let $k,k'\in \mathbb{Z}$ be such that $z[kp_i,(k+1)p_i)=z[k'p_i,(k'+1)p_i)$. Since $z[kp_i-n,kp_i+n)=z[-n,n)$ and $z[k'p_i-n,k'p_i+n)=z[-n,n)$, we have that $z[kp_i-n,kp_i)=z[k'p_i-n,k'p_i)$. Similarly, $z[(k+1)p_i-n,(k+1)p_i+n)=z[-n,n)$ and $z[(k'+1)p_i-n,(k'+1)p_i+n)=z[-n,n)$, so we have that $z[(k+1)p_i,(k+1)p_i+n)=z[(k'+1)p_i,(k'+1)p_i+n)$. In conclusion, by concatenating the two sequences, we have $z[kp_i-n,(k+1)p_i+n)=z[k'p_i-n,(k'+1)p_i+n)$. In particular for all $0\le s<p_i$, $z[kp_i+s-n,kp_i+s+n)=z[k'p_i+s-n,k'p_i+s+n)$. Take $kp_i+s$ and $k'p_i+s$ to be $m$ and $m'$ respectively, we have that  $d(f(z)(kp_i+s),f(z)(k'p_i+s))<\epsilon$.
        \end{proof}
         
         Recall that the canonical factor map for an Oxtoby system $(\overline{O}(z((p_i),(x_i))),S)$ is $\pi: \overline{O}(z((p_i),(x_i)))\rightarrow \varprojlim \mathbb{Z}_{p_i} $ such that $\pi(y)=(n_i)$ if and only if $y$ has the same $p_i$-skeleton with $S^{n_i}z$ for all $i$. The following definition is new.

         \begin{definition}\label{DF1}
             Let $X$ be a compact metric space, $(x_i)$ be a sequence of elements in $X$ and $(p_i)$ be a fast growing sequence.  Write $\pi: \overline{O}(z((p_i),(x_i)))\rightarrow \varprojlim \mathbb{Z}_{p_i} $ for the canonical factor map. Let $y\in \overline{O}(z((p_i),(x_i))$ and $\pi(y)=(n_i)$.  %We cut $y$ from $-n_i$ into consecutive intervals of length $p_i$, $\{y[kp_i-n_i,(k+1)p_i-n_i)\}_k$. By Lemma \ref{CS} and the definition of canonical factor map, for all $k\in \mathbb{Z}$, $y[kp_i-n_i,(k+1)p_i-n_i)$ contains elements of period $p_i$ (a subword in the $p_i$-skeleton of $y$) and only one other element which is not $p_i$ periodic.%
             Let $i<j$ be natural numbers. We call an interval $[kp_i-n_i,(k+1)p_i-n_i)$ a \textbf{$p_j$-$p_i$-piece} of $y$ if for all $n\in [kp_i-n_i,(k+1)p_i-n_i)$ we have that if $n\not \in {\rm Per}_{p_i}(y)$, then $n\in {\rm Per}_{p_j}(y)\setminus {\rm Per}_{p_{j-1}}(y)$.
         \end{definition}

        In the following lemma, we show that for every $k,i$, every interval of the form $[kp_i-n_i,(k+1)p_i-n_i)$ a $p_j$-$p_i$-piece of $y$ for some $j$. We also show  that $j$ is unique, although the latter will not be used.

         \begin{lemma}\label{UNI}
             Let $X$ be a compact metric space, $(x_i)$ be a sequence of elements in $X$ and $(p_i)$ be a fast growing sequence, $z=z((p_i),(x_i))$ be an Oxtoby sequence and let $\pi$ be the canonical factor map of $(\overline{O}(z((p_i),(x_i)),S)$. Now suppose $y\in \overline{O}(z)$ is a Topelitz sequence such that $\pi(y)=(n_i)$. Then for all $k\in \mathbb{Z}$ and $i\in \mathbb{N}$, there is a $j>i$ such that $[kp_i-n_i,(k+1)p_i-n_i)$ is a $p_j$-$p_i$-piece of $y$ and $j$ is unique.
         \end{lemma}

         \begin{proof}
             Fix $k$ and $i$. Let $j$ be the least natural number such that 
             \begin{equation}\label{UNI1}
                 [kp_i,(k+1)p_i)\subset {\rm Per}_{p_j}(z)
             \end{equation} 
             such $j$ exists since $z$ is Toeplitz and the fact that $(p_i)$ is a fast growing sequence. For all $i\leq j'<j$, since $p_i\mid p_{j'}$, we can find $k'$ such that $[kp_i,(k+1)p_i)\subset [k'p_{j'},(k'+1)p_{j'})$. By the choice of $j$, we have $[kp_i,(k+1)p_i)\not \subset {\rm Per}_{p_{j'}}(z)$. Thus, by Lemma \ref{OL1}, we have 
             \begin{equation}\label{UNI2}
                 [kp_i,(k+1)p_i)\cap J(j',k')=J(i,k)
             \end{equation}
              By (\ref{UNI1}) and (\ref{UNI2}), we have $ J(i,k)\subset {\rm Per}_{p_j}(z)\setminus {\rm Per}_{p_{j-1}}(z)$. Thus, the lemma is true for $z$. 
             
             Now for $y\in \overline{O}(z)$, fix $k$ and $i$. Let $j$ be the least natural number such that 
             $$
             [kp_i-n_i,(k+1)p_i-n_i)\subset {\rm Per}_{p_j}(y).
             $$ 
             By the definition of canonical factor map, we know $y$ has the same $p_j$-skeleton as $S^{n_j}z$, thus, 
             $$
             y[kp_i-n_i,(k+1)p_i-n_i)=z[kp_i-n_i+n_j,(k+1)p_i-n_i+n_j).
             $$
             Since $n_j\equiv n_i$ (mod $p_i$), the interval  
             $$
             [kp_i-n_i+n_j,(k+1)p_i-n_i+n_j)
             $$
             could be written as the form
             $$
             [k'p_t,(k'+1)p_t)
             $$
             for some $k'\in \mathbb{Z}$. Since the lemma is true for $z$, we know $[kp_i-n_i,(k+1)p_i-n_i)$ is a $p_j$-$p_i$-piece of $y$. The lemma is also true for $y$. 
         \end{proof}

         \begin{lemma}\label{SUB}
             Let $X$ be a compact metric space, $(x_i)$ be a sequence of elements in $X$ and $(p_i)$ be a fast growing sequence, $z=z((p_i),(x_i))$ be an Oxtoby sequence  Let $y\in \overline{O}(z)$, then for any $l>i\in \mathbb{N}$, the set ${\rm Per}_{p_l}(y)\setminus {\rm Per}_{p_{l-1}}(y)$ is a subset of the union of all $p_l$-$p_i$-pieces of $y$.
         \end{lemma}
         \begin{proof}
             Let $\pi$ be the canonical factor map of $(\overline{O}(z((p_i),(x_i)),S)$ and $\pi(y)=(n_i)$. Suppose $n$ is not in the union of all $p_l$-$p_i$-pieces of $y$. Since $\mathbb{Z}=\bigcup_k[kp_i-n_i,(k+1)p_i-n_i)$, by Lemma \ref{UNI}, $n$ belongs to a $p_j$-$p_i$-piece of $y$ for some $j\not =l$. By the definition of $p_j$-$p_i$-piece of $y$, we have that $n\in {\rm Per}_{p_i}(y)$ or $n\in {\rm Per}_{p_j}(y)\setminus {\rm Per}_{p_{j-1}}(y)$. The latter implies that $n\not \in {\rm Per}_{p_l}(y)\setminus {\rm Per}_{p_{l-1}}(y)$.
         \end{proof}

         For instance, let $z=z((p_i),(x_i))$ be an Oxtoby sequence, by the construction of Oxtoby sequence, given $i\geq 1$, all $p_{i+1}$-$p_i$-pieces of $z$ are $\{[-p_i+kp_{i+1},p_i+kp_{i+1})\}_{k\in \mathbb{Z}}$ (see Figure 2).

\begin{figure}[h!]
\centering
\begin{tikzpicture}
\usetikzlibrary{math}

\usetikzlibrary{decorations.pathreplacing}

\node (x0) at (-0.6, 0.0)    {} ;
\node (x0) at (-1, 0.0)    {} ;
\node (x0) at (-1.4, 0.0)    {} ;
\node (00) at ( 0, 0.0)    {} ;
\node (01) at ( 0.4, 0.0)    {} ;

\node (04) at ( 2.2, 0.0)    {$p_{i+1}$-$p_i$-piece} ;

\node (07) at ( 2.8, 0.0)    {} ;

\node (10) at ( 4.6, 0.0)    {$p_{i+1}$-$p_i$-piece} ;

\node (13) at (7 , 0.0)    {$p_{i+2}$-$p_i$-piece} ;

\node (16) at ( -0.2, 0.0)    {$p_{i+2}$-$p_i$-piece} ;

\node (20) at ( 7.4, 0.0)    {} ;
\node (20) at ( -1.8, 0.0)    {$\cdots$} ;
\node (20) at ( 8.6, 0.0)    {$\cdots$} ;
\node      at (3.4, -1.15) {0}       ;
\node      at (5.8, -1.15) {$p_i$}       ;
\node      at (1, -1.15) {$-p_i$}       ;
\node      at (8.2, -1.15) {$2p_i$} ;
\node      at (-1.4, -1.15) {$-2p_i$};
\draw[-] (-3, 0.3) -- (25 * 0.4, 0.3) ;
\draw[-] (-3, -0.3) -- (25 * 0.4, -0.3) ;

\draw[-] (3.4, -0.3) -- (3.4, 0.3) ;
\draw[-] (5.8, -0.3) -- (5.8, 0.3) ;
\draw[-] (-1.4, -0.3) -- (-1.4, 0.3);
\draw[-] (8.2, -0.3) -- (8.2, 0.3) ;
\draw[-] (1, -0.3) -- (1, 0.3) ;
\draw[->] (3.4, -0.3-0.5) -- (3.4, 0.3-0.75) ;
\draw[->] (5.8, -0.3-0.5) -- (5.8, 0.3-0.75) ;
\draw[->] (1, -0.3-0.5) -- (1, 0.3-0.75) ;
\draw[->] (8.2, -0.8) -- (8.2, -0.45) ;
\draw[->] (-1.4, -0.8) -- (-1.4, -0.45);
  
\end{tikzpicture}
\caption{$p_j$-$p_i$-pieces of $z$}

\end{figure}
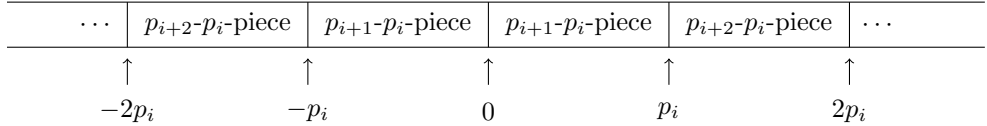

        \begin{lemma}\label{L7}
             Let $X$ be a compact metric space, $(x_i)$ be a sequence of elements in $x$. Let $(\overline{O}(z((p_i),(x_i)),S)$ be an Oxtoby system. For all $i<j\in \mathbb{N}$ there exists $l\in \mathbb{Z}$ such that $[lp_i,(l+1)p_i)$ is a $p_j$-$p_i$-piece of $z((p_i),(x_i))$.
        \end{lemma}
        
        \begin{proof}
           Write $z$ for $z((p_i),(x_i))$. Note that $\pi(z)=(0)$, so we need to find an $l\in \mathbb{Z}$ such that for all $n\in [lp_i,(l+1)p_i)$, if $n\not \in {\rm Per}_{p_i}(z)$, then $n\in {\rm Per}_{p_j}(z)\setminus {\rm Per}_{p_{j-1}}(z)$.  For every $0\leq k< \frac{p_{j-1}}{p_i}$ we apply Lemma \ref{OL1} with $k'=0$ and obtain that $[kp_i,(k+1)p_i)\subset {\rm Per}_{p_{j-1}}(z)$ or $[kp_i,(k+1)p_i)\cap J(j-1,0)=J(i,k)$. Since the sequence $(p_i)$ is fast growing, we know that $ [0,p_{j-1}) \not \subset {\rm Per}_{p_{j-1}}(z)$, so there exists $0\leq l< \frac{p_{j-1}}{p_i}$ such that $[lp_i,(l+1)p_i)\cap J(j-1,0)=J(i,l)$. Let $n\in [lp_i,(l+1)p_i)$ and suppose $n \not \in {\rm Per}_{p_i}(z)$. Since $[lp_i,(l+1)p_i)\subset [0,p_{j-1})$ and $[0,p_{j-1})\subset {\rm Per}_{p_j}(z)$, we have $n\in {\rm Per}_{p_j}(z)$. But $[lp_i,(l+1)p_i)\cap J(j-1,0)=J(i,l)$ implies that $[lp_i,(l+1)p_i)$ is disjoint with ${\rm Per}_{p_{j-1}}(z)\setminus {\rm Per}_{p_i}(z)$.  Thus $n\not \in {\rm Per}_{p_{j-1}}(z)$, we are done.
             
        \end{proof}
       
        \begin{lemma} \label{L6}
            Let $X$ be a compact metric space, $(x_i)$ be a sequence of elements in $x$. Let $y$ be an element in an Oxtoby system $(\overline{O}(z((p_i),(x_i)),S)$. Let $\pi$ be the canonical factor map of $(\overline{O}(z((p_i),(x_i)),S)$. Suppose $\pi(y)=(n_i)$, then for all $i< j$, there exists $l\in \mathbb{N}$ such that $[lp_{i}-n_{j},(l+1)p_{i}-n_{j})$ is a $p_j$-$p_i$-piece of $y$.
        \end{lemma}
        \begin{proof}
                 First,  by the definition of a $p_j$-$p_i$-piece in order to show that $[lp_{i}-n_{j},(l+1)p_{i}-n_{j})$ is a $p_j$-$p_i$-piece of $y$ we need to show that $[lp_{i}-n_{j},(l+1)p_{i}-n_{j})$ can be written as $[dp_{i}-n_{i},(d+1)p_{i}-n_{i})$ for some $d\in \mathbb{Z}$. The latter follows from the fact that $n_i\equiv n_j$ mod $p_i$.

                  By Lemma \ref{L7}, for all $i<j$, there exists $l\in \mathbb{Z}$ such that $[lp_{i},(l+1)p_{i})$ is a $p_j$-$p_{i}$-piece of $z$. By the definition of the shift map, we know 
              $$
              S^{n_j}z[lp_{i}-n_j,(l+1)p_{i}-n_j)=z[lp_{i},(l+1)p_{i}).
              $$ 
              Since $y$ has the same $p_{j}$-skeleton as $S^{n_{j}}z$, we have
              $$
              y[lp_{i}-n_{j},(l+1)p_{i}-n_{j})=S^{n_j}z[lp_{i}-n_{j},(l+1)p_{i}-n_{j}).
              $$
            Thus, $[lp_{i}-n_{j},(l+1)p_{i}-n_{j})$ is a $p_j$-$p_{i}$-piece of $y$.
        \end{proof}
            Now let us do some computation about $p_j$-$p_i$-pieces for $j=i+1$ and $i+2$.
            \begin{lemma}\label{I+1}
                  Let $X$ be a compact metric space, $(x_i)$ be a sequence of elements in $x$. Let $y$ be an element in an Oxtoby system $(\overline{O}(z((p_i),(x_i)),S)$. Let $\pi$ be the canonical factor map of $(\overline{O}(z((p_i),(x_i))),S)$. Suppose $\pi(y)=(n_i)$ and let $I$ be the interval $[mp_i-n_i,(m+1)p_i-n_i)$ of length $p_i$, then for all $i\in \mathbb{N}$, the following are equivalent:
                  
                  \begin{enumerate}
                      \item $I$ is a $p_{i+1}$-$p_i$-piece of $y$,
                      \item  there exists $k\in \mathbb{Z}$ such that
                      \begin{itemize}
                          \item $I= [kp_{i+1}-n_{i+1},kp_{i+1}-n_{i+1}+p_i)$ or
                          \item $I= [kp_{i+1}-n_{i+1}-p_i,kp_{i+1}-n_{i+1})$.
                      \end{itemize}  
                  \end{enumerate} 
                \end{lemma}

                \begin{proof}
                    Write $z=z((p_i),(x_i))$. Since $y$ has the same $p_{i+1}$-skeleton as $S^{n_{i+1}}z$, it is enough to show the following are equivalent:
                   
                     \begin{enumerate}
                      \item[(1')]$I$ is a $p_{i+1}$-$p_i$-piece of $S^{n_{i+1}}z$,
                      \item[(2')] there exists $k\in \mathbb{Z}$ such that
                      \begin{itemize}
                          \item$I= [kp_{i+1}-n_{i+1},kp_{i+1}-n_{i+1}+p_i)$ or
                          \item $I= [kp_{i+1}-n_{i+1}-p_i,kp_{i+1}-n_{i+1})$.
                     \end{itemize}  
                  \end{enumerate}
                    
                    In other words, we need to show the following are equivalent:
                    \begin{enumerate}
                      \item[(1'')] $[mp_i,(m+1)p_i)$ is a $p_{i+1}$-$p_i$-piece of $z$,
                      \item[(2'')]  there exists $k\in \mathbb{Z}$ such that
                      \begin{itemize}
                          \item $[mp_i,(m+1)p_i)=[kp_{i+1},kp_{i+1}+p_i)$ or 
                          \item $[mp_i,(m+1)p_i)= [kp_{i+1}-p_i,kp_{i+1})$. 
                          \end{itemize}  
                     \end{enumerate}
                
                Note that  by the definition of the Oxtoby sequence, (1'') is equivalent with
                    \begin{enumerate}
                        \item [(1*)]$J(i,m)\subset {\rm Per}_{p_{i+1}}(z)$.
                    \end{enumerate}

                      By the definition of the Oxtoby sequence,  (1*) is equivalent with $m \equiv 0$ or $-1$ mod $\frac{p_{i+1}}{p_i}$. The latter is equivalent to (2''). 
                \end{proof}
                \begin{lemma}\label{I+2}
                    Let $X$ be a compact metric space, $(x_i)$ be a sequence of elements in $X$. Let $y$ be an element in an Oxtoby system $(\overline{O}(z((p_i),(x_i))),S)$. Let $\pi$ be the canonical factor map of $(\overline{O}(z((p_i),(x_i))),S)$. Suppose $\pi(y)=(n_i)$, then for all $i\in \mathbb{N}$, let $I$ be the interval $[mp_i-n_i,(m+1)p_i-n_i)$ of length $p_i$, the following are equivalent:
                    \begin{enumerate}
                         \item  $I$ is a $p_{i+2}$-$p_i$-piece of $y$.
                         \item  there exists $k\in \mathbb{Z}$ such that
                         \begin{itemize}
                             \item $I\subset [kp_{i+2}-n_{i+2}+p_{i},kp_{i+2}-n_{i+2}-p_i+p_{i+1})$
                             or
                             \item $I\subset [kp_{i+2}-n_{i+2}+p_i-p_{i+1},kp_{i+2}-n_{i+2}-p_i)$.
                         \end{itemize}
                    \end{enumerate} 
                \end{lemma}
                \begin{proof}
               
                    Write $z$ for $z((p_i),(x_i))$. Since $y$ has the same $p_{i+2}$-skeleton as $S^{n_{i+2}}z$, it is enough to show the following are equivalent
                    \begin{enumerate}
                      \item[(1')]$I$ is a $p_{i+2}$-$p_i$-piece of $S^{n_{i+2}}z$,
                      \item[(2')] there exists $k\in \mathbb{Z}$ such that
                      \begin{itemize}
                          \item $I\subset [kp_{i+2}-n_{i+2}+p_{i},kp_{i+2}-n_{i+2}-p_i+p_{i+1})$
                          or
                             \item $I\subset [kp_{i+2}-n_{i+2}+p_i-p_{i+1},kp_{i+2}-n_{i+2}-p_i)$
                     \end{itemize}  
                  \end{enumerate}
                    
                    In other words, we need to show the following are equivalent:
                    \begin{enumerate}
                      \item[(1'')] $[mp_i,(m+1)p_i)$ is a $p_{i+2}$-$p_i$-piece of $z$,
                      \item[(2'')]  there exists $k\in \mathbb{Z}$ such that
                      \begin{itemize}
                          \item $[mp_i,(m+1)p_i)\subset [kp_{i+2}+p_{i},kp_{i+2}-p_i+p_{i+1})$ or 
                          \item $[mp_i,(m+1)p_i)\subset   [kp_{i+2}+p_i-p_{i+1},kp_{i+2}-p_i)$.
                     \end{itemize}  
                  \end{enumerate} 
                  Note that by the definition of the Oxtoby sequence, we have $(1'')$ is equivalent with
                  \begin{enumerate}
                      \item[(1*)] $J(i,m) \subset {\rm Per}_{p_{i+2}}(z)\setminus {\rm Per}_{p_{i+1}}(z)$.
                  \end{enumerate}

                  \begin{claim}
                      (1*) is equivalent with 
                  \begin{enumerate}
                      \item[(1**)] $\exists l\in \mathbb{Z}\quad J(i,m) \subset J(i+1,l) \subset {\rm Per}_{p_{i+2}}(z).$
                  \end{enumerate}
                  \end{claim}

                  \begin{proof}
                      ((1**)$\Rightarrow$ (1*)): This follows from the definition of  the Oxtoby sequence and the fact that for all $l\in \mathbb{Z}$, the set $J(i+1,l)$ is disjoint with ${\rm Per}_{p_{i+1}}(z)$. 
                   
                  ((1*)$\Rightarrow$ (1**)): First, since $p_i\mid p_{i+1}$ we can find a unique $l\in \mathbb{Z}$ such that $[mp_i,(m+1)p_i)\subset [lp_{i+1},(l+1)p_{i+1})$. By Lemma \ref{OL1}, $J(i,m)\subset {\rm Per}_{p_{i+1}}(z)$ or $J(i,m)=[mp_i,(m+1)p_i)\cap J(i+1,l)$. Since $J(i,m)$ is disjoint with ${\rm Per}_{p_{i+1}}(z)$, the former is impossible thus the latter must be true, hence we have 
                  \begin{equation}\label{FT}
                       J(i,m) \subset J(i+1,l).
                  \end{equation} Since $p_{i+1}\mid p_{i+2}$, we can find $l'\in \mathbb{Z}$, such that $[lp_{i+1},(l+1)p_{i+1})\subset [l'p_{i+2},(l'+1)p_{i+2})$ and by Lemma \ref{OL1}, we have $J(i+1,l)\subset {\rm Per}_{p_{i+2}}(z)$ or $J(i+1,l)= J(i+2,l')\cap [lp_{i+1},(l+1)p_{i+1})$. 
                  
                  \textbf{Case 1}: $J(i+1,l)\subset {\rm Per}_{p_{i+2}}$. This implies (1**). 
                  
                  \textbf{Case 2}: $J(i+1,l)= J(i+2,l')\cap [lp_{i+1},(l+1)p_{i+1})$. By (\ref{FT}), this implies $J(i,m)$ is a subset of $J(i+2,l')$ which is disjoint with ${\rm Per}_{p_{i+2}}(z)$. This contradicts (1*).
                  \end{proof}

                  %By the definition of the Oxtoby sequence, we know that $J(i,m)\subset {\rm Per}_{p_{i+1}}(z)$ if and only if $m\equiv 0$ or $-1$ mod %$\frac{p_{i+1}}{p_i}$. 
                  We will prove that (1**) implies (2'') and (2'') implies (1'').
                  
                  (1**)$\Rightarrow$ (2''): Let $l$ be such that $J(i,m)\subset J(i+1,l)\subset {\rm Per}_{p_{i+2}}(z)$. First, note that $J(i,m)\subset J(i+1,l)$ implies %[mp_i,(m+1)p_i)\subset [lp_{i+1},(l+1)p_{i+1})$. %Take $m=k$, $l=k'$ and $j=i+1$ in Lemma \ref{OL1}, we have $[mp_i,(m+1)p_i)\subset {\rm Per}_{p_{i+1}}(z)$ or $J(i,m)=J(i+1,l)\cap [mp_i,mp_{i+1})$. Since $J(i,m)\subset J(i+1,l)$, by the definition, we have $J(i+1,l)$ is disjoint with ${\rm Per}_{p_{i+1}}(z)$ and $J(i,m)\subset[mp_i,(m+1)p_i)$, thus $J(i,m)=J(i+1,l)\cap[mp_i,(m+1)p_i)$. thus we have  
                   \begin{equation}\label{I2}
                        [mp_i,(m+1)p_i)\subset [lp_{i+1},(l+1)p_{i+1}).
                   \end{equation}
                   By the definition of the Oxtoby sequence and $J(i+1,l)\subset {\rm Per}_{p_{i+2}}(z)$, we have $l \equiv 0$ or $l \equiv-1$ mod $\frac{p_{i+2}}{p_{i+1}}$. Thus, there exists $k\in \mathbb{Z}$ such that $l=k\frac{p_{i+2}}{p_{i+1}}$ or $l=k\frac{p_{i+2}}{p_{i+1}}-1$. Thus, by (\ref{I2}), there exists $k\in \mathbb{Z}$ such that
                  \begin{equation}\label{I3}
                      [mp_i,(m+1)p_i)\subset [k\frac{p_{i+2}}{p_{i+1}}p_{i+1},(k\frac{p_{i+2}}{p_{i+1}}+1)p_{i+1})=[kp_{i+2},kp_{i+2}+p_{i+1}]
                  \end{equation}
                  or
                  \begin{equation}\label{I4}
                      [mp_i,(m+1)p_i)\subset [(k\frac{p_{i+2}}{p_{i+1}}-1)p_{i+1},((k\frac{p_{i+2}}{p_{i+1}}-1)+1)p_{i+1})=[kp_{i+2}-p_{i+1},kp_{i+2}].
                  \end{equation}
                   By $J(i,m)\subset J(i+1,l)$ and the fact that $J(i+1,l)$ is disjoint with ${\rm Per}_{p_{i+1}}(z)$, we have that $[mp_i,(m+1)p_i)$ is not a $p_{i+1}$-$p_i$-piece of $z$. By Lemma \ref{I+1}, this implies that for all $n\in \mathbb{Z}$, we have $$
                   [mp_i,(m+1)p_i)\not =[np_{i+1},np_{i+1}+p_i)\quad \mbox{and} \quad [mp_i,(m+1)p_i)\not = [np_{i+1}-p_i,np_{i+1}).
                   $$
                  Since $p_i\mid p_{i+1}$, we have
                   \begin{equation}\label{I5}
                       [mp_i,(m+1)p_i)\cap[np_{i+1},np_{i+1}+p_i)=\emptyset \quad \mbox{and} \quad [mp_i,(m+1)p_i)\cap [np_{i+1}-p_i,np_{i+1})=\emptyset.
                   \end{equation}
                   \textbf{Case 1}: Suppose (\ref{I3}) holds. In (\ref{I5}), taking $n=k\frac{p_{i+2}}{p_{i+1}}$ in the first formula and $n=k\frac{p_{i+2}}{p_{i+1}}+1$ in the second formula, we obtain
                   $$
                   [mp_i,(m+1)p_i)\cap[kp_{i+2},kp_{i+2}+p_i)=\emptyset \quad \mbox{and} \quad [mp_i,(m+1)p_i)\cap [kp_{i+2}-p_i+p_{i+1},kp_{i+2}+p_{i+1})=\emptyset.
                   $$
                    Thus, $[mp_i,(m+1)p_i)$ is a subset of 
                   $$
                   [kp_{i+2},kp_{i+2}+p_{i+1})\setminus ([kp_{i+2},kp_{i+2}+p_i)\cup [kp_{i+2}-p_i+p_{i+1},kp_{i+2}+p_{i+1}))=[kp_{i+2}+p_i,kp_{i+2}-p_i+p_{i+1}).
                   $$ 
                   \textbf{Case 2}: Suppose (\ref{I4}) holds. In (\ref{I5}), taking $n=k\frac{p_{i+2}}{p_{i+1}}$ in the second formula and $n=k\frac{p_{i+2}}{p_{i+1}}-1$ in the first formula,
                   we have 
                    $$
                   [mp_i,(m+1)p_i)\cap[kp_{i+2}-p_i,kp_{i+2})=\emptyset \quad \mbox{and} \quad [mp_i,(m+1)p_i)\cap [kp_{i+2}-p_{i+1},kp_{i+2}+p_i-p_{i+1})=\emptyset.
                   $$
                   Thus, $[mp_i,(m+1)p_i)$ is a subset of 
                  $$
                  [kp_{i+2}-p_{i+1},kp_{i+2}]\setminus [kp_{i+2}-p_{i+1},kp_{i+2}+p_i-p_{i+1})\cup [kp_{i+2}-p_i,kp_{i+2}))=[kp_{i+2}+p_i-p_{i+1},kp_{i+2}-p_i).
                  $$

                  (2'')$\Rightarrow$(1''): When $[mp_i,(m+1)p_i)$ is a subset of
                  $$
                  I_1=[kp_{i+2}+p_{i},kp_{i+2}-p_i+p_{i+1})
                  $$ 
                  or 
                  $$
                  I_2=[kp_{i+2}+p_i-p_{i+1},kp_{i+2}-p_i).
                  $$
                  for some $k$. Since
                  $$
                  [kp_{i+2}+p_{i},kp_{i+2}-p_i+p_{i+1})\subset [kp_{i+2},kp_{i+2}+p_{i+1})
                  $$ 
                  and 
                  $$
                  [kp_{i+2}+p_i-p_{i+1},kp_{i+2}-p_i)\subset [kp_{i+2}-p_{i+1},kp_{i+2}).
                  $$
                  Both $I_1$ and $I_2$ are subsets of $[lp_{i+1},(l+1)p_{i+1})$ where $l \equiv 0$ or $l \equiv-1$ mod $\frac{p_{i+2}}{p_{i+1}}$. By the definition of the Oxtoby sequence, $[mp_i,(m+1)p_i)\subset {\rm Per}_{p_{i+2}}(z)$. By Lemma \ref{I+1}, we know $[mp_i,(m+1)p_i)$ is disjoint with all $p_{i+1}$-$p_i$-pieces. Thus, $[mp_i,(m+1)p_i)$ is not a $p_{i+1}$-$p_i$-piece. In conclusion, $[mp_i,(m+1)p_i)$ is a $p_{i+2}$-$p_i$-piece.
                  
                 \end{proof}
         \begin{lemma}\label{L3}
             Let $X$ be a compact metric space, $(x_i)$ be a sequence of elements in $X$ and $(p_i)$ be a fast growing sequence. Let $y$ be an element in an Oxtoby system $(\overline{O}(z((p_i),(x_i))),S)$ and $n\in {\rm Per}_{p_{i}}(y)\setminus {\rm Per}_{p_{i-1}}(y)$. Then for all $i<j \in \mathbb{N}$, there exists $k$ such that $y(n+kp_{i-1})=x_j$ and $n+kp_{i+1}\in {\rm Per}_{p_j}(y)\setminus {\rm Per}_{p_{j-1}}(y)$.
             
         \end{lemma}
         \begin{proof}

               Write $\pi$ for the maximal equicontinuous factor of  $(\overline{O}(z((p_i),(x_i))),S)$. Let $\pi(y)=(n_i)$.
              
              By Lemma \ref{L6}, for all $j\geq i$, there exists $l\in \mathbb{Z}$ such that $[lp_{i-1}-n_j,(l+1)p_{i-1}-n_j)$ is a $p_j$-$p_{i-1}$-piece of $y$. There exists $k\in \mathbb{Z}$, such that 
              \begin{equation}\label{EL3}
                  lp_{i-1}-n_{j}\leq n+kp_{i-1}<(l+1)p_{i-1}-n_{j}.
              \end{equation}

             By our assumption, $n\not \in {\rm Per}_{p_{i-1}}(y)$. Thus, $n+kp_{i-1}\not \in {\rm Per}_{p_{i-1}}(y)$. By (\ref{EL3}), $n+kp_{i-1}$ is in a $p_j$-$p_{i-1}$-piece of $y$, so we have $n+kp_{i-1}\in {\rm Per}_{p_j}(y)\setminus {\rm Per}_{p_{j-1}}(y)$. The latter implies that $y(n+kp_{i-1})=x_j$  by the construction of an Oxtoby sequence.
             \end{proof}

    \begin{lemma}\label{ToptoTop}
     Let $X$ be a compact metric space, $(x_i)$ and $(x'_i)$ be two sequences of elements in $X$ and $(p_i)$ be a fast growing sequence. Suppose both $L((x_i))$ and $L((x'_i))$ have at least two points and $f$ is a conjugacy map from $(\overline{O}(z((p_i),(x_i)),S)$ to $(\overline{O}(z,(p_i),(x'_i)),S)$, then $f$ maps Toeplitz sequences to Toeplitz sequences. 
\end{lemma} 

\begin{proof}
    Write $z$ for $z((p_i),(x_i))$ and $z'$ for $z((p_i),(x'_i))$. Let $\pi$ and $\pi'$ be the canonical factor map of $\overline{O}(z)$ and $\overline{O}(z')$, respectively. Since $L((x_i))$ has at least two points, by Lemma \ref{CS} and Lemma \ref{EQF}, $y\in \overline{O}(z)$ is Toeplitz if and only if $\pi^{-1}(\pi(y))$ is a singleton. For the same reason, $y'\in \overline{O}(z')$ is Toeplitz if and only if $\pi'^{-1}(\pi'(y'))$ is a singleton. By the property of isomorphism and equicontinuous factor, $\pi^{-1}(\pi(y))$ is a singleton implies $\pi'^{-1}(\pi'(f(x)))$ is a singleton. Thus, $f$ maps Toeplitz sequences to Toeplitz sequences.
\end{proof}
            
 \begin{lemma} \label{L4}
             Let $(X,d)$ be a compact metric space, $(x_i)$ and $(x'_i)$ be two sequences of elements in $X$ and $(p_i)$ be a fast growing sequence. Let $f$ be an isomorphism from $(\overline{O}(z((p_i),(x_i)),S)$ to $(\overline{O}(z((p_i),(x_i')),S)$. Suppose $L((x_i'))$ have at least two points. Then there exists $i_0\in \mathbb{N}$, such that for all $i\geq i_0$, $k\in \mathbb{Z}$ and $i<j$, we have $[kp_i,(k+1)p_i)\subset {\rm Per}_{p_j}(z)$ implies $[kp_i,(k+1)p_i) \subset {\rm Per}_{p_j}(f(z))$.
\end{lemma}
        \begin{proof}
            %By Lemma \ref{UNI}, for all $i\in \mathbb{N}$, $[kp_i,(k+1)p_i)$ is a $p_j$-$p_i$-piece of $z$ for some $j>i$ for all $k$. 
            
            Write $z$ for $z((p_i),(x_i))$ and $z'$ for $z((p_i),(x'_i))$. Let $a,b$ be two different points in $L((x_i'))$ and $\epsilon=\frac{d(a,b)}{2}>0$. Take $i_0$ given by Lemma \ref{L2} for $\epsilon$. We claim this $i_0$ works. We prove it by contradiction.
            
           Suppose $i_0$ does not work, then there exist $i_0<i<j$ and $k\in \mathbb{Z}$ such that $[kp_i,(k+1)p_i)\subset {\rm Per}_{P_{j}}(z)$ but $[kp_i,(k+1)p_i) \not \subset {\rm Per}_{p_{j}}(f(z))$. The latter implies there is $0\leq m \leq p_i-1$ such that $kp_i+m\not \in {\rm Per}_{p_j}(f(z))$. By the construction of the Oxtoby sequence, there exists $l>j$ such that $f(z)(kp_i+m)=x_l'$. 
           \begin{claim}\label{C1}
               For any $l_1>l$, there exists $k_1$ such that $f(z)(kp_i+m+k_1p_{j})=x'_{l_1}$.
           \end{claim}
           \begin{proof}
               By Lemma \ref{L3}, for any $l_1>l$ there exists $k'$ such that $f(z)(kp_i+m+k'p_{l-1})=x'_{l_1}$. Since $j<l$, we have $p_{j}\mid p_{l-1}$. Thus, we can rewrite $kp_i+m+k'p_{l-1}$ as $kp_i+m+k_1p_{j}$ for some $k_1\in \mathbb{Z}$.
           \end{proof}

            We can find two natural numbers $l_1,l_2$ greater than $l$ such that $d(x_{l_1}',a)<\frac{d(a,b)}{4},\, d(x_{l_2}',b)<\frac{d(a,b)}{4}$. So $d(x_{l_1}',x_{l_2}')>\epsilon$. Since $[kp_i,(k+1)p_i)\subset {\rm Per}_{p_j}(z)$, we have 
            $$
            z[kp_i+kp_{j},(k+1)p_i+kp_{j})
            $$ 
            are the same for all $k$. But in $f(z)$, by Claim \ref{C1}, there are $k_1,k_2\in \mathbb{Z}$ such that
            $$
            f(z)(kp_i+m+k_1p_{j})=x'_{l_1}
            $$
            and
            $$
            f(z)(kp_i+m+k_2p_{j})=x'_{l_2}
            $$
            Thus,
            $$
            d(f(z)(kp_i+m+k_1p_{j}),f(z)(kp_i+m+k_2p_{j}))=d(x'_{l_1},x'_{l_2})>\epsilon
            $$
            which contradicts the $i_0$ we choose.
        \end{proof}

       Now we prove Theorem \ref{Main1}:

        \vspace{0.5em}
        
           \noindent \textit{Proof of Theorem \ref{Main1}.} Write $z$ for $z((p_i),(x_i))$ and $z'$ for $z((p_i),(x'_i))$. Denote by $\pi$ the canonical factor map of $(\overline{O}(z),S)$. Let $\pi(f(z))=(n_i)$. We will prove that $(n_i)$ or $(p_i-n_i)$ is eventually a constant $m$, thus $f(z)=S^mz'$ or $S^{-m}z'$. Take $i_0$ for $z$ and $f(z)$ in Lemma \ref{L4}. We will prove that  $(n_i)_{i\geq i_0+2}$ or $(p_{i}-n_{i})_{i\geq i_0+2}$ is constant.

            \begin{lemma}\label{ML1}
                For all $i\ge i_0$, we have $n_{i+1}<p_{i}$ or $p_{i+1}-n_{i+1}<p_i$.
            \end{lemma}
                
            \begin{proof}
                Fix $i\geq i_0$ and suppose towards contradiction that $p_i\leq n_{i+1} \leq p_{i+1}-p_i$. By Lemma \ref{I+1}, the $p_{i+1}$-$p_i$-pieces of $f(z)$ are the intervals
                $$
                [-p_i-n_{i+1}+kp_{i+1},p_i-n_{i+1}+kp_{i+1})
                $$ 
                for $k\in \mathbb{Z}$. 
                \begin{claim}\label{C3}
                    The interval $[-p_i,0)$ or the interval $[0,p_i)$ is disjoint with
                $$
                \bigcup_{k\in \mathbb{Z}}[-p_i-n_{i+1}+kp_{i+1},p_i-n_{i+1}+kp_{i+1}).
                $$
                \end{claim}
               \begin{proof}
                   Denote by $I_k$ the interval $[-p_i-n_{i+1}+kp_{i+1},p_i-n_{i+1}+kp_{i+1})$. Note that the distance between the right endpoint of $I_k$ and the left endpoint of $I_{k+1}$ is $p_{i+1}\geq 3p_i$. Since the length of interval $[-p_i,p_i)=[-p_i,0)\cup[0,p_i)$ is equal to $2p_i$, it is enough to show for every $k\in \mathbb{Z}$, the interval $I_k$ intersects at most one of the intervals $[-p_i,0)$ or $[0,p_i)$. 
                   
                   When $k=0$, the interval $[-p_i-n_{i+1}+kp_{i+1},p_i-n_{i+1}+kp_{i+1})$ is equal to $[-p_i-n_{i+1},p_i-n_{i+1})$, and it is disjoint with $[0,p_i)$ since  $p_i<n_{i+1}$.  
                   
                   When $k=1$, the interval $[-p_i-n_{i+1}+kp_{i+1},p_i-n_{i+1}+kp_{i+1})$ is equal to $[-p_i-n_{i+1}+p_{i+1},p_i-n_{i+1}+p_{i+1})$, and it is disjoint with $[-p_i,0)$ since $n_{i+1}\leq p_{i+1}-p_i$. 
                   
                   Now suppose $k\geq 2$. Since $(p_i)$ is a fast growing sequence, we have $p_{i+1}\geq 3p_i$, and our assumption implies that $p_{i+1}-n_{i+1}\geq p_i$, which together implies the following
                   $$
                   -p_i-n_{i+1}+kp_{i+1}\geq  -p_i-n_{i+1}+2p_{i+1}=(p_{i+1}-p_i)+(p_{i+1}-n_{i+1})\geq 2p_i.
                   $$ In particular, the left endpoint of  the interval $[-p_i-n_{i+1}+kp_{i+1},p_i-n_{i+1}+kp_{i+1})$ is greater than $p_i$, so the interval is disjoint with both $[-p_i,0)$ and $[0,p_i)$. 
                   
                   Finally, suppose $k\leq -1$. Since $p_{i+1}\geq p_i$, and by definition $n_{i+1}\geq 0$ which together implies
                   $$
                   p_i-n_{i+1}+kp_{i+1}\leq p_i-p_{i+1}\leq -2p_i.
                   $$ 
                   In particular, the right endpoint of  the interval $[-p_i-n_{i+1}+kp_{i+1},p_i-n_{i+1}+kp_{i+1})$ is less than $-p_i$, so the interval is disjoint with both $[-p_i,0)$ and $[0,p_i)$.
               \end{proof}
                 
                 By Claim \ref{C3} and  Lemma \ref{SUB}, one of the intervals $[-p_i,0)$ and $[0,p_i)$ is not a subset of ${\rm Per}_{p_{i+1}}(f(z))$. %Since $f(z)$ is not periodic, there exists an $m\in [-p_i,p_i)$ such that $m\not \in {\rm Per}_{p_{i+1}}(f(z))$.%  
                 But $[-p_i,p_{i})\subset {\rm Per}_{p_{i+1}}(z)$ which contradicts the choice of $i\geq i_0$.
            \end{proof}
          
           % We will prove that either $n_{i+2}=n_{i+1}$ or $p_{i+2}-n_{i+2}=p_{i+1}-n_{i+1}$ for $i\geq i_0$.% 
           
           By Lemma \ref{ML1}, there are two cases: 

            \textbf{Case 1.} $n_{i_0+1}<p_{i_0}$.  In this case, we show that $(n_{i+1})_{i\geq i_0}$ is constant. We prove by induction on $i\geq i_0$ that $n_{i+1}=n_{i_0+1}$. The base case with $i=i_0$ is clear. Now we assume that $n_{i+1}=n_{i_0+1}$ and we will show that $n_{i+2}=n_{i+1}$. 
            By Lemma \ref{ML1}, we have that $n_{i+2}<p_{i+1}$ or $p_{i+2}-n_{i+2}<p_{i+1}$. Since $n_{i+2}\equiv n_{i+1}$ mod $p_{i+1}$, the former case implies that $n_{i+1}=n_{i+2}$, in which case we are done. We will reach a contradiction from the assumption that the latter case holds, in which case we have 
            
            \begin{equation}\label{MW}
                p_{i+2}-n_{i+2}=p_{i+1}-n_{i+1}. 
            \end{equation}

           Also, the following inequality follows from our assumption and the fact that $(p_i)$ is increasing.
           \begin{equation}\label{MW2}
               n_{i+1}<p_i.
           \end{equation}
            Now we look at the following interval
           
            $$
            D=[p_i-n_{i+1}-p_{i+1},-p_i-n_{i+1}).
            $$

            By Lemma \ref{I+2}, all $p_{i+2}$-$p_i$-pieces of $f(z)$ are $I_k=[kp_{i+2}-n_{i+2}+p_{i},kp_{i+2}-p_i-n_{i+2}+p_{i+1})$ and $J_k=[kp_{i+2}-n_{i+2}+p_i-p_{i+1},kp_{i+2}-n_{i+2}-p_i)$.

            \begin{claim}\label{MAINc}
            
                \begin{enumerate}
                    \item If $k\geq 1$, then $D$ is disjoint with $I_k$ and $J_k$.
                    \item If $k\leq 0$, then $D$ is disjoint with $I_k$ and $J_k$.
                   \end{enumerate}
            \end{claim}
            \begin{proof}
                (1) Since we have $p_{i+2}\geq 3p_{i+1}$, by (\ref{MW}), the left endpoint of $J_1$ is
            $$
            p_{i+2}-n_{i+2}+p_i-p_{i+1}= p_{i+1}-n_{i+1}+p_i-p_{i+1}=p_i-n_{i+1}>-p_i-n_{i+1}
            $$
            which is the left endpoint of $D$.
              The statement follow from the fact that $J_k<I_k<J_{k+1}$ for all $k\in \mathbb{Z}$.
             
             (2) Since we have $p_{i+2}\geq 3p_{i+1}$ and $n_{i+1}<p_{i+1}$, by (\ref{MW}), the right endpoint of $I_0$ is
            $$
           -p_i-n_{i+2}+p_{i+1}=-p_i-p_{i+2}+2p_{i+1}-n_{i+1}\leq-p_i-3p_{i+1}+2p_{i+1}-n_{i+1}=-p_i-n_{i+1}-p_{i+1}<p_i-n_{i+1}-p_{i+1}
            $$
            which is the left endpoint of $D$. Thus, the statement follows from the fact that $J_k<I_k<J_{k+1}$ for all $k\in \mathbb{Z}$.
             \end{proof}
            Claim \ref{MAINc} and Lemma \ref{I+2} imply that $D$ is disjoint with all $p_{i+2}$-$p_i$-pieces of $f(z)$.

            \begin{claim}\label{Mains2}
                $D$ is disjoint with any $p_{i+1}$-$p_i$-piece of $f(z)$
            \end{claim}
            \begin{proof}
                By Lemma \ref{I+1}, all $p_{i+1}$-$p_i$-pieces are of the form 
            $$
            P_k=[kp_{i+1}-n_{i+1},kp_{i+1}-n_{i+1}+p_i)\,\,\, \mbox{or}\,\,\, Q_k= [kp_{i+1}-n_{i+1}-p_i,kp_{i+1}-n_{i+1}).
            $$
            
            Note that the left endpoint of $D$ is the right endpoint of $Q_{-1}$, and the right endpoint of $D$ is
            $$
            -p_i-n_{i+1}<-n_{i+1}
            $$
            which is the left endpoint of $P_0$. The statement follows from the fact that
            $P_k<Q_k<P_{k+1}$ for all $k$.
            \end{proof}

           Since $n_{i+1} \equiv n_i$ mod $p_i$ and $p_i\mid p_{i+1}$, we can write 
           $$
           D=[ap_i-n_i,bp_i-n_i)
           $$
           where $a,b \in \mathbb{Z}$. Thus, by Lemma \ref{UNI}, there exists $j>i$ such that $D$ contain a $p_j$-$p_i$-piece of $f(z)$. Thus, by Claim \ref{MAINc} and Claim \ref{Mains2}, there exists $j>i+2$ such that $D$ contains a $p_{j}$-$p_i$-piece of $f(z)$ with $j>i+2$. Fix such a $j>i+2$.
            
            By (\ref{MW2}), the left endpoint of $D$ is
            $$
            p_i-n_{i+1}-p_{i+1}>-p_{i+1}.
            $$
            Thus, we have
           $$
           D\subset[-p_{i+1},0).
           $$
           
           Note that $[-p_{i+1},0)\subset {\rm Per}_{p_{i+2}}(z)$ by definition of the Oxtoby sequence. But since $D$ contains a $p_j$-$p_i$-piece of $f(z)$, we have $[-p_{i+1},0)\not \subset {\rm Per}_{p_{i+2}}(f(z))$. This contradicts the $i_0$ we choose.
            
           \textbf{Case 2} $p_{i_0+1}-n_{i_0+1}<p_{i_0}$. In this case, we show that $(p_{i+1}-n_{i+1})_{i\geq i_0}$ is constant. We prove by induction on $i\geq i_0$ that $p_{i+1}-n_{i+1}=p_{i_0+1}-n_{i_0+1}$. The base case with $i=i_0$ is clear. Now we assume that $p_{i+1}-n_{i+1}=p_{i_0+1}-n_{i_0+1}$ and we will show that $p_{i+2}-n_{i+2}=p_{i+1}-n_{i+1}$. By Lemma \ref{ML1}, we have $n_{i+2}<p_{i+1}$ or $p_{i+2}-n_{i+2}<p_{i+1}$. Since $n_{i+2}\equiv n_{i+1}$ mod $p_{i+1}$, the latter case implies that $p_{i+1}-n_{i+1}=p_{i+2}-n_{i+2}$, in which case we are done. We will reach a contradiction from the assumption that the former case holds, in which case we have 
           \begin{equation}\label{MW3}
               n_{i+2}=n_{i+1}.
           \end{equation} 
              Also, the following inequality follows from our assumption and the fact $p_i$ is increasing
              \begin{equation}\label{MW4}
                  p_{i+1}-n_{i+1}<p_i.
              \end{equation}
              Now we look at the following interval. 
            $$
            F=[p_i-n_{i+1}+p_{i+1},-p_i-n_{i+1}+2p_{i+1}).
            $$
           \begin{claim}\label{MAINc1}
                 \begin{enumerate}
                    \item If $k\geq 1$, then $F$ is disjoint with $I_k$ and $J_k$.
                    \item If $k\leq 0$, then $F$ is disjoint with $I_k$ and $J_k$.
                   \end{enumerate}
            \end{claim}
            \begin{proof}
                (1) Since we have $p_{i+2}\geq 3p_{i+1}$, by (\ref{MW3}), the left endpoint of $J_1$ is
            $$
            p_{i+2}-n_{i+2}+p_i-p_{i+1}= p_{i+2}-n_{i+1}+p_i-p_{i+1}\geq 3p_{i+1}-n_{i+1}+p_i-p_{i+1}> -p_i-n_{i+1}+2p_{i+1}
            $$
            which is the left endpoint of $F$.
              The statement follow from the fact that $J_k<I_k<J_{k+1}$ for all $k\in \mathbb{Z}$.
             
             (2) Since we have $p_{i+2}\geq 3p_{i+1}$ and $n_{i+1}<p_{i+1}$, by (\ref{MW3}), the right endpoint of $I_0$ is
            $$
           -p_i-n_{i+2}+p_{i+1}=-p_i-n_{i+1}+p_{i+1}<p_i-n_{i+1}+p_{i+1}
            $$
            which is the left endpoint of $F$. Thus, the statement follows from the fact that $J_k<I_k<J_{k+1}$ for all $k\in \mathbb{Z}$.
             \end{proof}
            Claim \ref{MAINc} and Lemma \ref{I+2} imply that $F$ is disjoint with all $p_{i+2}$-$p_i$-pieces of $f(z)$.

            \begin{claim}\label{Mains3}
                $F$ is disjoint with any $p_{i+1}$-$p_i$-piece of $f(z)$
            \end{claim}
            \begin{proof}
                By Lemma \ref{I+1}, all $p_{i+1}$-$p_i$-pieces are of the form 
            $$
            P_k=[kp_{i+1}-n_{i+1},kp_{i+1}-n_{i+1}+p_i)\,\,\, \mbox{or}\,\,\, Q_k= [kp_{i+1}-n_{i+1}-p_i,kp_{i+1}-n_{i+1}).
            $$
            
            Note that the left endpoint of $F$ is the right endpoint of $P_{1}$, and the right endpoint of $F$ is the left endpoint of $Q_2$. The statement follows from the fact that
            $P_k<Q_k<P_{k+1}$ for all $k$.
            \end{proof}
            Since $n_{i+1} \equiv n_i$ mod $p_i$ and $p_i\mid p_{i+1}$, we can write 
           $$
           F=[ap_i-n_i,bp_i-n_i)
           $$
           where $a,b \in \mathbb{Z}$. Thus, by Lemma \ref{UNI}, there exists $j>i$ such that $F$ contain a $p_j$-$p_i$-piece of $f(z)$. Thus, by Claim \ref{MAINc} and Claim \ref{Mains2}, there exists $j>i+2$ such that $F$ contains a $p_{j}$-$p_i$-piece of $f(z)$ with $j>i+2$. Fix such a $j>i+2$.

            By (\ref{MW4}), the right endpoint of $F$ is
            $$
            -p_i-n_{i+1}+2p_{i+1}=(p_{i+1}-n_{i+1}-p_i)+p_{i+1}<p_{i+1}.
            $$
            Thus, we have
           $$
           F\subset[0,p_{i+1}).
           $$
           
           Note that $[0,p_{i+1})\subset {\rm Per}_{p_{i+2}}(z)$ by definition of the Oxtoby sequence. But since $F$ contains a $p_j$-$p_i$-piece of $f(z)$, we have $[0,p_{i+1})\not \subset {\rm Per}_{p_{i+2}}(f(z))$. This contradicts the $i_0$ we choose.
            \par\hfill $\square$

         \vspace{0.5em}

       \section{Reverse elements of Oxtoby systems}\label{proof 2}
        In this section, we state and prove one more result which can be of independent interest and we give another proof of Theorem \ref{Main1}.
        
         \begin{definition}
       For $x\in X^\mathbb{Z}$, $x^{-1}$ is the bi-infinite sequence such that $x^{-1}(n)=x(-n)$.
\end{definition} 

Note that if $x$ is a point in an Oxtoby system $\overline{O}(z((p_i),(x_i)))$, then so is $x^{-1}$.

        \begin{theorem} \label{MT1'}
     Let $(x_i)$ be a sequence of points in a compact metric space $X$, $(p_i)$ be a fast growing sequence, and $z=z((p_i),(x_i))$ be the Oxtoby sequence. Suppose $L((x_i))$ have at least two points and $x \in \overline{O}(z)$ is a Toeplitz sequence. Then the following statements are equivalent
     \begin{enumerate}
      \item  $x=S^mz$ for some $m\in \mathbb{Z}$.
         \item there is a conjugacy map $f$ from $\overline{O}(z)$ to itself such that $f(x)=x^{-1}$.
        
     \end{enumerate}
\end{theorem}

       %The property described in this theorem is invariant under isomorphisms, thus, together with Lemma \ref{ToptoTop}, this theorem implies Theorem \ref{Main1}.

Now we show a proof from Theorem 5.1 to Theorem 6.1.
\begin{proof}[Proof of Theorem \ref{Main1} from Theorem \ref{MT1'}] Write $z$ for $z((p_i),(x_i))$ and $z'$ for $z((p_i),(x'_i))$. By Lemma \ref{ToptoTop}, $f(z)$ is a Toeplitz sequence in $\overline{O}(z')$. For any integer sequence $(n_k)$, the sequence $S^{n_k}(f(z))$ converges if and only if the sequence $S^{n_k}(z)$ converges. By Theorem \ref{MT1'} with $x=z$, the sequence $S^{n_k}(z)$ converges if and only if 
the sequence $S^{n_k}(z^{-1})$ converges. By definition the sequence $S^{-n_k}(z)$ converges if and only if $S^{-n_k}(f(z))$ converges if and only if $S^{n_k}([f(z)]^{-1})$ converges. Hence $S^{n_k}(f(z))$ converges if and only if $S^{n_k}([f(z)]^{-1})$ converges. By Lemma \ref{min}, there is a conjugacy map $g$ from $\overline{O}(z')$ to itself such that $g(f(z))=[f(z)]^{-1}$. Then by Theorem \ref{MT1'} again, $f(z)$ is in the orbit of $z'$.\par\hfill 
\end{proof}

\begin{definition}
    Let $x \in X^{\mathbb{Z}}$ be a non-periodic sequence. Let $n_0\le n_1$ be integers, then the interval $[n_0,n_1]$ is called a \textbf{maximal $p$-periodic block} in $x$ if $[n_0,n_1]\subset {\rm Per}_p(x)$, and there is no interval $[n'_0,n'_1]\subset {\rm Per}_p(x)$ such that $[n_0,n_1]$ is a proper subset of $[n'_0,n'_1]$.  
    \end{definition}
    \vspace{0.5em}

    Throughout the next part of this section, given a fast growing sequence $(p_i)$, let 
    $$q_0=0 \quad{\rm and} \quad q_i=\sum_{0<j\le i}p_j$$ for $i>0$.

    \begin{lemma}\label{6_1}
    Let $(x_i)$ be a sequence of points in a compact metric space $X$, $(p_i)$ be a fast growing sequence, and $z=z((p_i),(x_i))$ be the Oxtoby sequence. Then for any $m \ge1$ and maximal $p_m$-periodic block $[n_0,n_1]$ in $z$, 
    \begin{enumerate}
        \item if there is $k\in\mathbb{Z}$ such that $kp_m\in[n_0,n_1]$, then $[n_0,n_1]=[kp_m-q_{m-1}-1,kp_m+q_{m-1}]$;
        \item if for every $k\in\mathbb{Z}$, $kp_m\notin[n_0,n_1]$, then $[n_0,n_1]$ is a maximal $p_{m-1}$-periodic block in $z$.
    \end{enumerate}
   
    \end{lemma}
    \begin{proof}

       Note by the definition of Oxtoby sequence, we have

\begin{equation}\label{eq6.1}
               {\rm Per}_{p_{m+1}}(z)\backslash{\rm Per}_{p_{m}}(z)\subset\bigcup_{k\in\mathbb{Z}}[kp_{m+1}-p_{m},kp_{m+1}+p_{m}].
           \end{equation} 
        
        \noindent(1) We prove (1) by induction on $m$. When $m=1$, by the definition of Oxtoby sequence, every maximal $p_1$-periodic block in $z$ is of the form $[kp_1-1,kp_1]=[kp_1-q_0-1,kp_1+q_0]$ for some $k\in\mathbb{Z}$. This show that (1) holds for $m=1$. 
        
       Now we perform the induction step from $m$ to $m+1$. 
        
        \vspace{0.5em}
 Fix $k\in\mathbb{Z}$ such that $kp_m\in[n_0,n_1]$. By the definition of the Oxtoby sequence, $$[kp_{m+1}-p_{m}-1,kp_{m+1}+p_{m}]\subset {\rm Per}_{p_{m+1}}(z).$$ By our induction hypothesis and the fact that $p_{m}|p_{m+1}$, the maximal $p_m$-periodic block containing $kp_{m+1}-p_{m}$ is 
 
 $$[kp_{m+1}-p_m-q_{m-1}-1,kp_{m+1}-p_m+q_{m-1}]=[kp_{m+1}-q_{m}-1,kp_{m+1}-p_m+q_{m-1}]\subset {\rm Per}_{p_{m}}(z)\subset {\rm Per}_{p_{m+1}}(z).$$
  Similarly, the maximal $p_m$-periodic block containing $kp_{m+1}+p_{m}$ is  $$[kp_{m+1}+p_m-q_{m-1}-1,kp_{m+1}+p_m+q_{m-1}]=[kp_{m+1}+p_m-q_{m-1}-1,kp_{m+1}+q_{m}]\subset {\rm Per}_{p_{m}}(z)\subset {\rm Per}_{p_{m+1}}(z).$$  Then we have,
  \begin{align}
 & [kp_{m+1}-q_{m}-1,kp_{m+1}-p_m+q_{m-1}]\cup[kp_{m+1}-p_{m}-1,kp_{m+1}+p_{m}]\cup[kp_{m+1}+p_m-q_{m-1}-1,kp_{m+1}+q_{m}]\nonumber\\
  &=[kp_{m+1}-q_{m}-1,kp_{m+1}+q_{m}]\subset {\rm Per}_{p_{m+1}}(z)\nonumber.
\end{align}

        Consider the positions $kp_{m+1}-q_{m}$ and $kp_{m+1}+q_{m}+1$, we need to show these two positions do not belong to ${\rm Per}_{p_{m+1}}(z)$. Because $[kp_{m+1}-q_{m}-1,kp_{m+1}-p_m+q_{m-1}]$ and $[kp_{m+1}+p_m-q_{m-1}-1,kp_{m+1}+q_m]$ are maximal $p_m$-periodic blocks in $z$, we have $kp_{m+1}-q_{m}\notin{\rm Per}_{p_{m}}(z)$ and $kp_{m+1}+q_{m}+1\notin{\rm Per}_{p_{m}}(z)$. By the definition of fast growing sequence, $$(k-1)p_{m+1}+p_{m}<kp_{m+1}-q_{m}<kp_{m+1}-p_{m},$$ and $$kp_{m+1}+p_{m}<kp_{m+1}+q_{m}+1<(k+1)p_{m+1}-p_{m}.$$ So $\bigcup_{k\in\mathbb{Z}}[kp_{m+1}-p_{m},kp_{m+1}+p_{m}]$ does not contain $kp_{m+1}-q_{m}$ or $kp_{m+1}+q_{m}+1$. Then by (\ref{eq6.1}) we have $kp_{m+1}-q_{m}\notin{\rm Per}_{p_{m+1}}(z)$ and $kp_{m+1}+q_{m}+1\notin{\rm Per}_{p_{m+1}}(z)$. As a result, $[kp_{m+1}-q_m-1,kp_{m+1}+q_m]$ is a maximal $p_{m+1}$-periodic block in $z$, $[kp_{m+1}-q_m-1,kp_{m+1}+q_m]=[n_0,n_1]$. This ends the induction step of the proof of (1).

        \vspace{0.5em}
\noindent(2) Fix a maximal $p_{m+1}$-periodic block $[n_0,n_1]$ in $z$ such that there is no $k\in\mathbb{Z}$ with $kp_{m+1}\in[n_0,n_1]$. For every $k\in\mathbb{Z}$, $kp_{m+1}\in {\rm Per}_{p_{1}}(z)\subset{\rm Per}_{p_{m+1}}(z)$, so by part (1), $[kp_{m+1}-q_{m}-1,kp_{m+1}+q_{m}]$ is a maximal $p_{m+1}$-periodic block in $z$. By definition, two different maximal $p_{m+1}$-periodic blocks in $z$ are disjoint. Then we get that $[n_0,n_1]$ is disjoint from $\bigcup_{k\in\mathbb{Z}}[kp_{m+1}-q_{m},kp_{m+1}+q_{m}]$. By (\ref{eq6.1}), $[n_0,n_1]$ is disjoint from ${\rm Per}_{p_{m+1}}(z)\backslash{\rm Per}_{p_{m}}(z)$. Then we have $[n_0,n_1]\subset{\rm Per}_{p_m}(z)$ by $[n_0,n_1]\subset{\rm Per}_{p_{m+1}}(z)$. By definition, a maximal $p_{m+1}$-periodic block in $z$ contained in ${\rm Per}_{p_m}(z)$ is a maximal $p_{m}$-periodic block in $z$, so $[n_0,n_1]$ is a maximal $p_{m}$-periodic block in $z$. This ends the proof of (2).
\end{proof}

\begin{lemma}\label{L'1}
    Let $(x_i)$ be a sequence of points in compact metric space $X$, $(p_i)$ be a fast growing sequence, $z=z((p_i),(x_i))$ be an Oxtoby sequence and $x\in\overline{O}(z)$ be a Toeplitz sequence. Then for any $m \ge1$ and maximal $p_m$-periodic block $[n_0,n_1]$ in $x$, there is $i<m$ so that $x[n_0,n_1]=z[-q_i-1,q_i]$.
    %\item for any maximal $p_m$-periodic block $[n_0,n_1]$ in $x$, there is $i<m$ so that $x[n_0,n_1]=z[-q_i-1,q_i]$. 
    %\item for any maximal $p_m$-periodic block $[n_0,n_1]$ in $x$, if there exists $i<m$ so that $x[n_0,n_1]=z[-q_i-1,q_i]$, then $[n_0,n_1]$ is a maximal $p_{i+1}$-periodic block in $x$. 
   % \item if $[n_0,n_1]$ and $[n'_0,n'_1]$ are maximal $p_m$-periodic blocks in $x$, and $i<m$ is such that $x[n_0,n_1]=x[n'_0,n'_1]=z[-q_i-1,q_i]$, then $n_0-n'_0$ is a multiple of $p_{i+1}$.
    
\end{lemma}  

\begin{proof}
   Note that for every $m\ge1$, the statement of the lemma depends only on the $p_m$-skeleton of $x$. Since $x$ has the same $p_m$-skeleton as $S^k(z)$ for some $k\in\mathbb{Z}$, and the result of this lemma is invariant under the shift action, without loss of generality, we can assume that $x=z$. 
   
   We prove this lemma by induction on $m$. When $m=1$, by the definition of Oxtoby sequence, every maximal $p_1$-periodic block in $z$ is of the form $[kp_1-1,kp_1]$ for some $k\in\mathbb{Z}$ and $z[kp_1-1,kp_1]=z[-q_0-1,q_0]$. This implies that the lemma holds for $m=1$.
   
   Now we perform the induction step from $m$ to $m+1$. Let $[n_0,n_1]$ be a maximal $p_{m+1}$-periodic block in $z$, we need to show there is $i\le m$ so that $z[n_0,n_1]=z[-q_i-1,q_i]$. There are two cases:
   
   \vspace{0.5em}
  \noindent{\textbf{Case 1.} There is $k\in\mathbb{Z}$ so that $kp_{m+1}\in[n_0,n_1]$;}
  \vspace{0.5em}

  In this case, by Lemma \ref{6_1}, $[n_0,n_1]=[kp_{m+1}-q_m-1,kp_{m+1}+q_m]$. Since $[n_0,n_1]\subset{\rm Per}_{m+1}(z)$, we have $z[n_0,n_1]=z[-q_m-1,q_m]$.

\vspace{0.5em}
  \noindent{\textbf{Case 2.} There is no $k\in\mathbb{Z}$ so that $kp_{m+1}\in[n_0,n_1]$.}
  \vspace{0.5em}

  By Lemma \ref{6_1}, $[n_0,n_1]$ is a maximal $p_{m}$-periodic block in $z$. Then by the induction hypothesis, there is $i<m$ so that $z[n_0,n_1]=z[-q_i-1,q_i]$.

  So the lemma holds for $m+1$. By induction, we are done.
 \end{proof}

\begin{lemma}\label{L'1_1}
    Let $(x_i)$ be a sequence of points in compact metric space $X$, $(p_i)$ is a fast growing sequence, $z=z((p_i),(x_i))$ be an Oxtoby sequence and $x\in\overline{O}(z)$ be a Toeplitz sequence. Then for any $m \ge1$ and maximal $p_m$-periodic block $[n_0,n_1]$ in $x$, if there exists $i<m$ so that $x[n_0,n_1]=z[-q_i-1,q_i]$, then $[n_0,n_1]$ is a maximal $p_{i+1}$-periodic block in $x$.
    \end{lemma}

    \begin{proof}
      Note that for every $m\ge1$, the statement of the lemma depends only on the $p_m$-skeleton of $x$. Since $x$ has the same $p_m$-skeleton as $S^k(z)$ for some $k\in\mathbb{Z}$, and the result of this lemma is invariant under the shift action, without loss of generality, we can assume that $x=z$. 

      We prove this lemma by induction on $m$. When $m=1$, the result is trivial. Suppose this lemma holds for some $m\ge1$. Given a maximal $p_{m+1}$-periodic block $[n_0,n_1]$ in $z$,we have two cases:

\vspace{0.5em}
  \noindent{\textbf{Case 1.} There is $k\in\mathbb{Z}$ so that $kp_{m+1}\in[n_0,n_1]$;}
  \vspace{0.5em}

  By Lemma \ref{6_1}, $z[n_0,n_1]=z[-q_m-1,q_m]$, and $[n_0,n_1]$ is a maximal $p_{m+1}$-periodic block in $z$, there is nothing to prove.

  \vspace{0.5em}
  \noindent{\textbf{Case 2.} There is no $k\in\mathbb{Z}$ so that $kp_{m+1}\in[n_0,n_1]$;}
  \vspace{0.5em}

  By Lemma \ref{6_1}, $[n_0,n_1]$ is a maximal $p_{m}$-periodic block in $z$. Then by Lemma \ref{L'1}, there is $i<m$ so that $z[n_0,n_1]=z[-q_i-1,q_i]$, by the induction hypothesis, $[n_0,n_1]$ is a maximal $p_{i+1}$-periodic block in $z$.

  So the lemma holds for $m+1$. By induction, we are done.

    \end{proof}

\begin{lemma}\label{L'2}
   Let $(x_i)$ be a sequence of points in compact metric space $X$, $(p_i)$ is a fast growing sequence,  $z=z((p_i),(x_i))$ be an Oxtoby sequence and $x\in\overline{O}(z)$ be a Toeplitz sequence. Suppose $m_1\le m_2$, and both of the following hold:
   
   \begin{enumerate}
       \item the interval $[n_0,n_1]$ is a maximal $p_{m_1+1}$-periodic block of $x$ but not a maximal $p_{m_1}$-periodic block of $x$;
       \item the interval $[n_2,n_3]$ is a maximal $p_{m_2+1}$-periodic block of $x$ but not a maximal $p_{m_2}$-periodic block of $x$.
   \end{enumerate}
   Then $p_{m_1+1}$ divides both $n_2-n_0$ and $n_1-n_3$.
\end{lemma}

\begin{proof}
    %By the same reason in Lemma \ref{L'1}%, 
    
    Note that for every $m_1$ and $m_2$, the statement of the lemma depends only on the $p_{m_2+1}$-skeleton of $x$. Since $x$ has the same $p_{m_2+1}$-skeleton as $S^k(z)$ for some $k\in\mathbb{Z}$, and the result of this lemma is invariant under the shift action, without loss of generality, we can assume that $x=z$. 
    
     Since the interval $[n_0,n_1]$ does not satisfy the conclusion of Lemma \ref{6_1}(2) for $m=m_1+1$, by Lemma \ref{6_1}(1), we know that $n_0=kp_{m_1+1}-q_{m_1}-1$ for some $k\in\mathbb{Z}$. Similarly, $n_2=k'p_{m_2+1}-q_{m_2}-1$ for some $k'\in\mathbb{Z}$. Since we have $m_1\le m_2$, by the definition of fast growing sequence, $p_{m_1+1}$ divides $n_2-n_0$. By the same reason, $p_{m_1+1}$ divides $n_1-n_3$. This ends the proof.
\end{proof}

%Let $(x_i)$ be a sequence of points in a compact metric space $X$, $(p_i)$ be a fast growing sequence, and $z=z((p_i),(x_i))$ be the Oxtoby sequence. Suppose $L((x_i))$ have at least two points and $x \in \overline{O}(z)$ is a Toeplitz sequence. Then the following statements are equivalent
    % \begin{enumerate}
     %    \item there is a conjugacy map $f$ from $(\overline{O}(z((p_i),(x_i)))$ to itself such that $f(x)=x^{-1}$.
    %     \item  $x=S^mz$ for some $m\in \mathbb{Z}$.

\begin{proof}[Proof of Theorem \ref{MT1'}] 
\noindent(1)$\Rightarrow$(2) By the definition of Oxtoby sequence, for every $n\in\mathbb{Z}$, we have $z(n)=z(-n-1)$, so $z^{-1}=S^{-1}z$. When $x=S^nz$, then $x^{-1}=[S^n(z)]^{-1}=S^{-n}(z^{-1})=S^{-n-1}z$. So $S^{-2n-1}$ is a conjugacy map from $(\overline{O}(z)$ to itself so that $S^{-2n-1}(x)=x^{-1}$.

(2)$\Rightarrow$(1) Let $x\in \overline{O}(z)$ be a Toeplitz sequence and $f$ is a conjugacy map from $\overline{O}(z)$ to itself so that $f(x)=x^{-1}$. Since (1) and (2) of Theorem \ref{MT1'} are invariant under the shift action, without loss of generality, we can assume that $0\in {\rm Per}_{p_1}(x)$. For $i\ge1$, let $[-n_i,m_i]$ be the maximal $p_i$-periodic block in $x$ which contains the position 0, with $n_i,m_i\geq 0$. Since $x$ is a Toeplitz sequence, both $n_i$ and $m_i$ are increasing and go to $+\infty$ as $i$ goes to infinity. By Lemma \ref{L'1}, for every 
$i\ge1$, there is $t_i<i$ such that $x[-n_i,m_i]=z[-q_{t_i}-1,q_{t_i}]$. Thus we have $n_i+m_i=2q_{t_i}+1$ and $t_i\rightarrow\infty$. By Lemma \ref{L'1_1}, for every $i\ge1$, $[-n_i,m_i]$ is a maximal $p_{t_i+1}$-periodic block in $x$. And by Lemma \ref{L'1}, for every $i\ge1$, $[-n_i,m_i]$ is not a maximal $p_{t_i}$-periodic block in $x$. 

For the integer sequence $(n_{i}-m_{i})$, either it is bounded or it is unbounded. If it is bounded then it has a constant subsequence; if it is unbounded, it has a strictly monotonic subsequence. So we have three cases.

\noindent\textbf{Case 1}. There is a subsequence $(i_j)$ of natural numbers such that $(n_{i_j}-m_{i_j})$ is constant equal to $s_0$.\\

\begin{figure}[h!]
\centering
\begin{tikzpicture}
\usetikzlibrary{math}
\usetikzlibrary{decorations.pathreplacing}

\node (x0) at (-5.0, 0.0)    {$\cdot$} ;
\node (x0) at (-4.6 , 0.0)    {$\cdot$} ;
\node (x0) at (-4.2, 0.0)    {$\cdot$} ;
\node (00) at (-4, 0.6)     {$-n_{i_2}$} ;

\node  at (9.2, 0.6)     {$m_{i_2}$} ;
\node  at (0, 0.6)     {$-n_{i_1}$} ;
\node  at (5.2, 0.6)     {$m_{i_1}$} ;
\node (02) at (2 , 0.6)     {$-n_{i_0}$} ;
\node (03) at ( 3.2, 0.6)     {$m_{i_0}$} ;
\node (20) at (9.4, 0.0)    {$\cdot$} ;
\node (20) at (9.8, 0.0)    {$\cdot$} ;
\node (20) at (10.2, 0.0)    {$\cdot$} ;

\draw[-] (-15 * 0.4, 0.3) -- (30 * 0.4, 0.3) ;
\draw[-] (-15 * 0.4, -0.3) -- (30 * 0.4, -0.3) ;
\draw[-] (0, -0.3) -- (0, 0.3) ;
\draw[-] (5.2, -0.3) -- (5.2, 0.3) ;
\draw[-] (-4, -0.3) -- (-4, 0.3) ;

\draw[-] (3.2, -0.3) -- (3.2, 0.3) ;

\draw[-] (2, -0.3) -- (2, 0.3) ;

\draw[-] (9.2, -0.3) -- (9.2, 0.3) ;

\end{tikzpicture}
\caption{Case 1. }
\end{figure}

In this case, for every $j\ge0$, by $n_{i_j}+m_{i_j}=2q_{t_{i_j}}+1$, we have $n_{i_j}=q_{t_{i_j}}+\frac{s_0+1}{2}$. This means that $$x[-q_{t_{i_j}}-\frac{s_0+1}{2},q_{t_{i_j}}-\frac{s_0-1}{2}]=z[-q_{t_{i_j}}-1,q_{t_{i_j}}].$$ Since $q_{t_{i_j}}$ diverges to $+\infty$, we have $x=S^{\frac{s_0-1}{2}}(z).$

\noindent\textbf{Case 2.} There is a subsequence $(i_j)$ of natural numbers such that $(n_{i_j}-m_{i_j})$ is strictly increasing.

We will show that this case is impossible.

\begin{figure}[h!]
\centering
\begin{tikzpicture}
\usetikzlibrary{math}

\usetikzlibrary{decorations.pathreplacing}

\node (x0) at (-5.0, 0.0)    {$\cdot$} ;
\node (x0) at (-4.6 , 0.0)    {$\cdot$} ;
\node (x0) at (-4.2, 0.0)    {$\cdot$} ;
\node (00) at (-4, 0.6)     {$-n_{i_2}$} ;

\node  at (9.2, 0.6)     {$m_{i_2}$} ;
\node  at (2, 0.6)     {$-n_{i_1}$} ;
\node  at (7.2, 0.6)     {$m_{i_1}$} ;
\node (02) at (5.2 , 0.6)     {$-n_{i_0}$} ;
\node (03) at ( 6.0, 0.6)     {$m_{i_0}$} ;
\node (20) at (9.4, 0.0)    {$\cdot$} ;
\node (20) at (9.8, 0.0)    {$\cdot$} ;
\node (20) at (10.2, 0.0)    {$\cdot$} ;

\draw[-] (-6, 0.3) -- (12, 0.3) ;
\draw[-] (-6 , -0.3) -- (12, -0.3) ;
\draw[-] (5.2, -0.3) -- (5.2, 0.3) ;
\draw[-] (7.2, -0.3) -- (7.2, 0.3) ;
\draw[-] (-4, -0.3) -- (-4, 0.3) ;

\draw[-] (6.0, -0.3) -- (6.0, 0.3) ;

\draw[-] (2, -0.3) -- (2, 0.3) ;

\draw[-] (9.2, -0.3) -- (9.2, 0.3) ;

\end{tikzpicture}
\caption{Case 2. }
\end{figure}

 For every $j\ge0$, since $[-n_{i_j},m_{i_j}]$ is a maximal $p_{t_{i_j}+1}$-periodic block in $x$, we have that $[m_{i_j}-n_{i_j},2m_{i_j}]$ is a maximal $p_{t_{i_j}+1}$-periodic block in $S^{-m_{i_j}}(x)$, in particular, 
 \begin{equation}\label{eq6.2}
      [m_{i_j}-n_{i_j},2m_{i_j}]\subset {\rm Per}_{p_{t_{i_j}+1}}(S^{-m_{i_j}}(x)). 
 \end{equation}
 By Lemma \ref{L'2}, for any $j<j'$, $p_{t_{i_j}+1}$ divides $m_{i_{j'}}-m_{i_j}$, so $S^{-m_{i_{j'}}}(x)$ and $S^{-m_{i_j}}(x)$ have the same $p_{t_{i_j}+1}$-skeleton. Thus $S^{-m_{i_{j'}}}(x)$ and $S^{-m_{i_j}}(x)$ agree on every position in ${\rm Per}_{p_{t_{i_j}+1}}(S^{-m_{i_j}}(x))={\rm Per}_{p_{t_{i_j}+1}}(S^{-m_{i_{j'}}}(x))$. So by (\ref{eq6.2}), for every $j<j'$, we have   
 \begin{equation}\label{6.3}
    \forall k\in[m_{i_j}-n_{i_j},2m_{i_j}]\quad S^{-m_{i_j}}(x)(k)=S^{-m_{i_{j'}}}(x)(k).  
 \end{equation}
  By the assumption that $(n_{i_j}-m_{i_j})$ is strictly increasing, $m_{i_j}-n_{i_j}$ diverges to $-\infty$ when $j$ goes to $+\infty$. So for every $k\in\mathbb{Z}$, by (\ref{6.3}), we have that $S^{-m_{i_j}}(x)(k)$ is eventually constant. In particular, $S^{-m_{i_j}}(x)$ converges to some point $y\in\overline{O}(z)$. By (\ref{eq6.2}) and (\ref{6.3}), we have that $[m_{i_j}-n_{i_j},2m_{i_j}]\subset {\rm Per}_{p_{t_{i_j}+1}}(y)$ for every $j\ge0$, so $y$ is a Toeplitz sequence. Since $f$ is a conjugacy map from $\overline{O}(z)$ to itself such that $f(x)=x^{-1}$ and  $S^{-m_{i_j}}(x)$ converges to $y$, we have that $S^{-m_{i_j}}(x^{-1})$ converges to $f(y)$. However, since for any $j\ge0$, the interval $[-n_{i_j},m_{i_j}]$ is a maximal $p_{t_{i_j}}$-periodic block in $x$, we have that $[-m_{i_j},n_{i_j}]$ is a maximal $p_{t_{i_j}}$-periodic block in $x^{-1}$ and $[0,m_{i_j}+n_{i_j}]$ is a maximal $p_{t_{i_j}}$-periodic block in $S^{-m_{i_j}}(x^{-1})$. As a result, 
  \begin{equation}\label{eq6.3}
     \forall j\ge0\quad -1\notin {\rm Per}_{p_{t_{i_j}}}(S^{-m_{i_j}}(x^{-1})). 
 \end{equation}
 
 By Lemma \ref{L'2}, for any $j<j'$, $p_{t_{i_j}+1}$ divides $m_{i_{j'}}-m_{i_j}$, so $S^{-m_{i_{j'}}}(x^{-1})$ and $S^{-m_{i_j}}(x^{-1})$ have the same $p_{t_{i_j}+1}$-skeleton. By (\ref{eq6.3}), we have that 
 \begin{equation}\label{6}
     \forall j'>j\ge0\quad -1\notin {\rm Per}_{p_{t_{i_j}}}(S^{-m_{i_{j'}}}(x^{-1})). 
 \end{equation}
  Then by (\ref{6}), $-1\notin {\rm Per}_{p_{t_{i_j}+1}}(f(y))$ for any $j\ge0$. In other words $-1\notin{\rm Per}(f(y))$, so $f(y)$ is not a Toeplitz sequence. This contradicts Lemma \ref{ToptoTop}. The contradiction shows that this case is impossible.

\noindent\textbf{Case 3.} There is a subsequence $(i_j)$ of natural numbers such that $(n_{i_j}-m_{i_j})$ is strictly decreasing.\

We will show that this case is also impossible.

In this case, $x^{-1}\in \overline{O}(z)$ is a Toeplitz sequence and $f^{-1}$ is a conjugacy map from $\overline{O}(z)$ to itself so that $f^{-1}(x^{-1})=x=(x^{-1})^{-1}$. Also, for every $i\ge1$, $[-m_i,n_i]$ is a maximal $p_i$-periodic block in $x^{-1}$. So $x^{-1}$ and $f^{-1}$ satisfy Case 2, but Case 2 is impossible, a contradiction.

\end{proof}

        \vspace{0.5em}    
    \section{ The equivalence relation $E_{{\rm tt}}([0,1]^{\mathbb{N}})$} \label{Borel reduction}
    
          In this section, we look at the topological type equivalence relation defined in Section \ref{TT relation} and connect it to conjugacy of Oxtoby systems. %By Corollary 5.13, we know that in general, two Oxtoby systems ($\overline{O}(z((p_i),(x_i)),S)$ and $(\overline{O}(z'((p_i),(x'_i)),S)$ are isomorphic if and only if $(x_i)$ and $(x_i')$ have the same topological type. 

            \begin{theorem}\label{oxtoby}
                
           Let $(X,d)$ be a compact metric space, $(x_i)$ and $(x_i')$ be two sequences in $X$. Fix a fast growing sequence $(p_i)$. Then the following are equivalent
           \begin{enumerate}
               \item The Oxtoby systems $(\overline{O}(z((p_i),(x_i))),S)$  and $(\overline{O}(z'((x_i'),(p_i))),S)$ are conjugate.
               \item $(x_i)$ and $(x_i')$ have the same topological type. 
           \end{enumerate} 
     \end{theorem}

            \begin{proof}
            Write $z$ for $z((p_i),(x_i))$ and $z'$ for $z'((x_i),(p_i))$ and let $\pi$ be the canonical factor map of $\overline{O}(z)$.
            
            \textbf{Case 1}: Suppose one of $L(x_i)$ and $L(x_i')$ is singleton, without loss of generality, we assume $L(x_i)$ is a singleton. %We will show that $\pi$ is injective. By Lemma \ref{EQF}(ii), $\pi$ is injective on Topelitz sequences. Let $x$ be the unique element in $L((x_i))$.  If $y$ is a non-Toeplitz sequence in $\overline{O}(z((p_i),(x_i)))$, then by Lemma \ref{Will}, for every $n\in {\rm Aper}(y)$, we have $y(n)=x$.  Since ${\rm Per}(y)$ and the restriction of $y$ on ${\rm Per}(y)$ is uniquely determined by $\pi(y)$, this shows $\pi^{-1}(\pi(y))$ is a singleton. Thus, $(\overline{O}(z((p_i),(x_i))),S)$ is conjugate with its maximal equicontinuous factor.
          By Proposition \ref{SIN}, the system $(\overline{O}(z),S)$ is conjugate to its maximal equicontinuous factor.

            (1) $\Rightarrow$ (2) Since the systems $(\overline{O}(z),S)$ and $(\overline{O}(z'),S)$ are conjugate, $(\overline{O}(z'),S)$ is also conjugate with its maximal equicontinuous factor which is the odometer system $(\varprojlim \mathbb{Z}_{p_i},\hat{1})$. We claim $L(x'_i)$ is also a singleton. Indeed suppose that $L(x'_i)$ has at least two points, then by Theorem \ref{Main1}, the only automorphisms of $(\overline{O}(z'),S)$ are shifts. But in an odometer system, any two points can be mapped by an automorphism, a contradiction. Thus, $L(x'_i)$ is also a singleton, which means both $(x_i)$ and $(x'_i)$ are convergent sequences, so they have the same topological type.

            (2) $\Rightarrow$ (1) Since $(x_i)$ and $(x'_i)$ have the same topological type, we have $L(x'_i)$ is also a singleton. By Proposition \ref{Will}, we have both systems $(\overline{O}(z'),S)$ and $(\overline{O}(z),S)$ are conjugate to their maximal equicontinuous factor $(\varprojlim \mathbb{Z}_{p_i},\hat{1})$.

       \textbf{Case 2}:  Suppose both $L(x_i)$ and $L(x'_i)$ have at least two points. %By Theorem \ref{Main1},
       %we need to show that, $(x_i)$ and $(x_i')$ have the same topological type if and only if there is an isomorphism from $(\overline{O}(z),S)$ to $(\overline{O}(z'),S)$ sending $z$ to $z'$.%

            (1)$\Rightarrow$ (2). Suppose $(x_i)$ and $(x_i')$ have different topological types. Without loss of generality, assume that there exists an increasing sequence $(i_k)$ such that $(x_{i_k})$ converges while $(x_{i_k}')$ diverges. By Lemma \ref{non-Toeplitz exist}, we can find a non-Topelitz sequence $y\in \overline{O}(z
            )$ and write $\pi(y)=(n_i)$. 
            %\begin{claim}\label{CC1}
               % Both $(n_i)$ and $(p_i-n_i)$ go to infinity.
            %\end{claim}
            %\begin{proof}
                %First suppose $(n_i)$ is eventually constant equal to $n$, then $y=S^nz$ which is Toeplitz, a contradiction. Now suppose $(p_i-n_i)$ does not converge to infinity. Since $n_i<p_i$, we have that $(p_i-n_i)$ is eventually constant, say equal to $n$, then $y=S^{-n}z$ which is also Toeplitz, a contradiction. 
            %\end{proof} 
            
            By Lemma \ref{Will1}(1), for every $m\in\mathbb{Z}$, there is an $k_0$ such that for all $k\geq k_0$ we have
            $$
            -n_{i_k}<m<p_{i_k}-n_{i_k}.
            $$
             We claim that the sequence $(S^{n_{i_k-1}}z)$ converges. To this end we need to show that the sequence $(S^{n_{i_k-1}}z)(m)$ converges for any $m\in \mathbb{Z}$. If $m\in {\rm Per}(y)$, then $(S^{n_{i_k-1}}z)(m)$ is eventually constant so converges. If $m\in {\rm Aper}(y)$ then we have $S^{n_{i_k-1}}z(m)=x_{i_k}$ by Lemma \ref{Will1}(2). By our assumption, $S^{n_{i_k-1}}z(m)=x_{i_k}$ converges as $k\rightarrow\infty$. On the other hand, we claim that the sequence $S^{n_{i_k-1}}z'$ does not converge. Indeed, $z'$ is constructed in the same way as $z$, given that $(n_i)$ belongs to the maximal equicontinuous factor of $(\overline{O}(z'),S)$ which is the same as the maximal equicontinuous factor of $(\overline{O}(z),S)$, we have $ {\rm Aper}((n_i))\neq \emptyset$. Now pick an $m\in {\rm Aper}((n_i))$. By Lemma \ref{Will1}(2), we have $S^{n_{i_k-1}}z'(m)=x'_{i_k}$  for all $k\geq k_0$. By our assumption $(x'_{i_k})$ diverges as $k$ goes to infinity, so $S^{n_{i_k-1}}z'$ diverges. Thus, $(S^nz)$ and $(S^nz')$ have different topological types. By Lemma \ref{min}, there is no isomorphism from $\overline{O}(z)$ to $\overline{O}(z')$ sending $z$ to $z'$ and by Theorem \ref{Main1},  $\overline{O}(z)$ and $\overline{O}(z')$ are not conjugate.

            (2)$\Rightarrow$(1). Suppose $(x_i)$ and $(x_i')$ have the same topological type. We will show that $(S^{n}z)$ and $(S^nz')$ have the same topological type. Suppose $S^{n_k}z$ converges to $y$ for an increasing $(n_k)$. The  ${\rm Per}_{p_i}(S^{n_k}z)$ and ${\rm Per}_{p_i}(S^{n_k}z')$ are the same for all $i\in \mathbb{N}$ since they are constructed from the same fast growing sequence. Hence, for all $m\in {\rm Per}(y)$, the sequence $S^{n_k}z'(m)$ is eventually constant. Now for any $m\in {\rm Aper}(y)$, if  $S^{n_k}z(m)=x_{i_k}$, then $S^{n_k}z'(m)=x'_{i_k}$, so, since $S^{n_k}z(m)$ converges, we have that $(x_{i_k})$ converges. But since $m\in {\rm Aper}(y)$, we have $i_k$ goes to infinity, so, $(x'_{i_k})$ converges. In conclusion, $S^{n_k}z'(m)$ converges for all $m$, so $(S^{n_k}z')$ also converges. The argument for $S^{n_k}z'$ converges implying $S^{n_k}z$ converges is symmetric. By Lemma \ref{min2}, $\overline{O}(z)$ and $\overline{O}(z')$ are conjugate. 
            \end{proof}   

           \vspace{0.5em}

            %\noindent \textit{Remarks}: When $L((x_i))$ is a singleton, then $(\overline{O}(z((x_i),(p_i))),S)$ is conjugate to its maximal equicontinuous factor $(G,\widehat{1})$ while other Oxtoby systems whose $\omega$-limit sets have two points are not. So Theorem \ref{Main1} is still true without the condition $L((x_i))$ has more than two points.
            
          \begin{theorem}\label{R}
              Let $(X,d)$ be a compact metric space. The topological type relation $E_{\rm tt}(X)$ is Borel reducible to the isomorphism relation of Oxtoby subshifts in $(X^{\mathbb{Z}},S)$. In particular, the topological type relation  $E_{\rm tt}([0,1]^{\mathbb{N}})$ is Borel reducible to the isomorphism relation of minimal compact systems.
          \end{theorem}
          \begin{proof}
              This follows from Theorem \ref{oxtoby}, we fix a fast growing sequence $(p_i)$ for Oxtoby sequences. We send a sequence $(x_i)$ in $X$ to the Oxtoby system $(\overline{O}(z((p_i),(x_i))$. By Theorem  \ref{oxtoby}, the topological type of $(x_i)$ determines the isomorphism type of $(\overline{O}(z),(x_i),(p_i))$, which ends the proof.
          \end{proof}
          
           %The next theorem implies $E_{\rm tt}(X)$ is a Borel equivalence relation since the isomorphism of pointed minimal system by the theorem of Kaya \cite[Theorem 1.1]{Kaya2}.
           
          \begin{theorem}\label{BR}
              The isomorphism relation of pointed minimal compact systems is Borel bireducible with the equivalence relation $E_{\rm tt}([0,1]^{\mathbb{N}})$.
          \end{theorem}
          \begin{proof}
              We first show $E_{\rm tt}([0,1]^{\mathbb{N}})$ is Borel reducible to the isomorphism relation of pointed minimal compact systems. We send a sequence $(x_i)$ in the Hilbert cube to 
              $$
              (\overline{O}(z((p_i),(x_i))),S,z((p_i),(x_i))).
              $$  
              The Borelness of the map is routine to check. By Theorem \ref{oxtoby}, the topological type of the sequence determines the isomorphism type of Oxtoby systems. Also, by Theorem \ref{Main1}, two Oxtoby systems are conjugate if and only if there is an isomorphism sending the Oxtoby sequence to the other Oxtoby sequence. Thus, this is a Borel reduction.
              
              Now we show the isomorphism relation of pointed minimal compact systems is Borel reducible to $E_{\rm tt}([0,1]^{\mathbb{N}})$. For a pointed minimal compact system $(X,f,x)$, we send it to the sequence $(x,f(x),x,f^2(x),x,f^3(x),...\dots)$, by Corollary \ref{BRBR}, this is a Borel reduction.
               
               %{\color{red}Suppose first two pointed minimal compact systems $(X,f,x)$ and $(Y,g,y)$ are conjugate by $h$. Then for any increasing natural number sequence $(n_k)$, if $(f^{n_k}(x))$ converges to $x_0$ in $X$, then $(g^{n_k}(y))$ converges to $h(x_0)$. Since $h(x)=y$, if $(f^{n_k}(x))$ converges to $x$, then $(h^{n_k})(y)$ converges to $h(x)=y$. This is enough to show that the two sequences $(x,f(x),x,f^2(x),\dots)$ and $(y,g(y),y,g^2(y),\dots)$ have the same topological type.
              
              %For the other direction, suppose  
              %$\theta((X,f,x))$ and $\theta((Y,g,y))$ have the same topological type. Then, as a subsequence in the same position, $(f^n(x))$ and $(g^n(y))$ have the same topological type. Also, take an increasing sequence of natural numbers $(n_k)$ such that the limit of $f^{n_k}(x)$ is $x$ then the limit of $g^{n_k}(y)$ is $y$. By Lemma \ref{min}, $(X,f,x)$ and $(Y,g,y)$ are isomorphic.}
          \end{proof}

          Now we prove $E_{{\rm tt}}([0,1]^{\mathbb{N}})$ is not Borel reducible to any $S_{\infty}$-action. We consider the action of $c_0$ on $(S^1)^{\mathbb{N}}$ defined by
          $$
          g\circ (x_n)= (g(n)+x_n)
          $$
          where $g\in c_0$ and $x_n\in S^1$ and denote this action by $E^{c_0}_{(S^1)^{\mathbb{N}}}$. This action is turbulent \cite[Theorem 10.5.2]{Gao} thus not classifiable by countable structures.
            
         \begin{theorem}\label{Turbulent}
         The equivalence relation $E_{\rm tt}([0,1]^{\mathbb{N}})$ is not Borel reducible to any $S_{\infty}$-actions.
         \end{theorem}
         \begin{proof}
              We first show that $E^{c_0}_{(S^1)^{\mathbb{N}}}$ is Borel reducible to $E_{{\rm tt}}(S^1)$. We fix a countable dense subset of $S^1$ as $(q_n)$. Then send a sequence $x=(x_n)$ in $(S^1)^{\mathbb{N}}$ to $y=(y_n)$ where $y_{2k}=x_k$ and $y_{2k+1}=q_k$ for all $k\in \omega$. It is routine to check this is a Borel reduction. 
         
             Now let $i$ be an embedding from $S^1$ to the Hilbert cube, it is straightforward to check the map sending a sequence $(x_n)$ in $(S^1)^{\mathbb{N}}$ to $(i(x_n))$ is a Borel reduction from $E_{{\rm tt}}(S^1)$ to $E_{{\rm tt}}([0,1]^{\mathbb{N}})$. Thus, $E^{c_0}_{(S^1)^{\mathbb{N}}}\leq_B E_{{\rm tt}}([0,1]^{\mathbb{N}})$ and since $E^{c_0}_{(S^1)^{\mathbb{N}}}$ is not Borel reducible to any Borel $S_{\infty}$-action by \cite[Theorem 10.5.2]{Gao}, this ends the proof.
         \end{proof}
          Now we can prove the main theorems in the paper.
        \begin{proof}[Proof of Theorems \ref{CM} and \ref{PCM}]
            This follows from Theorem \ref{R}, Theorem \ref{BR} and Theorem \ref{Turbulent}.
        \end{proof}

         \begin{proof}[Proof of Theorem \ref{FCM}]
              Note that Oxtoby systems are conjugate to their inverses, so two Oxtoby systems are flip conjugate if and only if they are conjugate. So the theorem follows from Theorem \ref{R} and Theorem \ref{Turbulent}.
         \end{proof} 

         \section{The equivalence relation $E_{\rm csc}$}\label{Ding Gu}
         In \cite{DG}, Ding and Gu introduced the following equivalence relation. 
         \begin{definition}
            Let $\mathbb{X}_{\rm cpt}$ be the set of all metrics $r$ on $\omega$ such that the completion of $(\omega,r)$ is compact. The equivalence relation $E_{\rm csc}$ on $\mathbb{X}_{\rm cpt}$ is defined as $rE_{\rm csc}s$ iff the set of Cauchy sequences in $(\omega,r)$ is same as that in $(\omega,s)$.
         \end{definition}
        The next theorem uses the notion of $Z$-sets, we refer the reader to \cite{Mill} or \cite{Chapman} for the definition.
        
        Denote by $Z([0,1]^{\mathbb{N}})$ the set of all sequences in the Hilbert cube whose closure is a $Z$-set and by ${\rm Homeo}([0,1]^{\mathbb{N}})$ the homeomorphism group of $([0,1]^{\mathbb{N}})$. The group ${\rm Homeo}([0,1]^{\mathbb{N}})$ act naturally on $Z([0,1]^{\mathbb{N}})$:
            $$
             f\circ (x_n)=(f(x_n)).
            $$
          \begin{theorem}\label{hilberttt}
              The equivalence relation $E_{\rm tt}([0,1]^{\mathbb{N}})$ is bi-reducible with the relation $E_{\mathrm{csc}}$.
          \end{theorem} 
     
     \begin{proof}
         In \cite[Theorem 4.2]{DG}, it is shown that the relation $E_{\mathrm{csc}}$ is Borel bi-reducible with the action of ${\rm Homeo}([0,1]^{\mathbb{N}})$ on $Z([0,1]^{\mathbb{N}})$ defined above. On the other hand, the relation $E_{\mathrm{tt}}([0,1]^{\mathbb{N}})$ is Borel bi-reducible with the the conjugacy of pointed minimal compact systems by Theorem \ref{BR}. Therefore it suffices to reduce the conjugacy of pointed minimal compact systems to the action of ${\rm Homeo}([0,1]^{\mathbb{N}})$ on $Z([0,1]^{\mathbb{N}})$ and then reduce the action of ${\rm Homeo}([0,1]^{\mathbb{N}})$ on $Z([0,1]^{\mathbb{N}})$ to $E_{\rm tt}([0,1]^{\mathbb{N}})$.

         To get the reduction from conjugacy of pointed minimal compact systems to the action of ${\rm 
  Homeo}([0,1]^{\mathbb{N}})$ on  $Z([0,1]^{\mathbb{N}})$, we fix a pointed minimal compact system $(X,\varphi,x)$. We view $X$ as a subspace of the Hilbert cube. Write $\iota\colon[0,1]^{\mathbb{N}}\to[0,1]^{\mathbb{N}}$ for an embedding of the Hilbert cube in itself as a $Z$-set. Using Lemma \ref{min}, it is routine to check that the map taking $(X,\varphi,x)$ to 
         $$
         (\iota(x),\iota(\varphi(x)),\iota(\varphi^2(x)),\ldots)\in Z([0,1]^{\mathbb{N}})
         $$
         is a Borel reduction from the conjugacy of pointed  minimal compact systems to the equivalence relation induced by the action of ${\rm Homeo}([0,1]^{\mathbb{N}})$ on $Z([0,1]^{\mathbb{N}})$.

         To find a reduction from the action of ${\rm 
  Homeo}([0,1]^{\mathbb{N}})$ on  $Z([0,1]^{\mathbb{N}})$ to $E_{\rm tt}([0,1]^{\mathbb{N}})$, we identify $[0,1]^{\mathbb{N}}$ with $[0,1]^{\mathbb{N}}\times[0,1]$. Given a sequence $(x_n)$ in $Z([0,1]^{\mathbb{N}})$, we map it to
         $$
         \theta((x_n))=((x_1,1),(x_1,\frac{1}{2}),(x_2,\frac{1}{3}),(x_1,\frac{1}{4}),(x_2,\frac{1}{5}),(x_3,\frac{1}{6})\ldots)
         $$
         (that is, at step $n+1$ we put the sequence \[(x_1,\frac{1}{n(n+1)\slash 2+1}),(x_2,\frac{1}{n(n+1)\slash 2+2}),\ldots,(x_{n+1},\frac{1}{(n+1)(n+2)
         \slash 2}))\] after $(x_{n},\frac{1}{n(n+1)\slash 2})$). If a homeomorphism of $[0,1]^{\mathbb{N}}$ maps $(x_n)$ to $(y_n)$, then clearly $\theta((x_n))$ and $\theta((y_n))$ have the same topological type.

          On the other hand, if $\theta(x_n)$ and $\theta(y_n)$ have the same topological type. Since $(\frac{1}{n})$ is a convergent sequence which will not influence the convergence of the first coordinate, in other words, the sequence
         $$
         x=(x_1,x_1,x_2,x_1,x_2,\dots)
         $$
         and 
         $$
         y=(y_1,y_1,y_2,y_1,y_2,\dots)
         $$
         have the same topological type. Now by Fact \ref{strong TT implies homeo}, the map sending the limit of any increasing convergent subsequence of $x(n)$ to the limit of the corresponding subsequence $y(n)$ is a homeomorphism between $X=\overline{\{x_n\mid n\in\mathbb{N}\}}$ and $Y=\overline{\{y_n\mid n\in\mathbb{N}\}}$ taking $x_i$ to $y_i$. Since both $X$ and $Y$ are $Z$-sets, by \cite[Theorem 6.3.4]{Mill} this homeomorphism can be extended to $[0,1]^{\mathbb{N}}$.
             \end{proof}

         \section{Other bounds on the complexity of the conjugation relation of minimal compact systems}\label{s9}

         First, we give an upper bound on the complexity in the Borel reducibility hierarchy of the conjugacy relation of minimal compact systems.      
               \begin{theorem}
                The conjugacy relation of minimal compact systems is Borel reducible to an orbit equivalence relation.
               \end{theorem}
               \begin{proof}
                   Let $D$ be the set of all closed subsystems of $(([0,1]^{\mathbb{N}})^{\mathbb{Z}},S)$. Note that $D$ is a Borel subset of $\mathcal{K}(([0,1]^{\mathbb{N}})^{\mathbb{Z}})
        $. For an automorphism $f\in \rm{Aut}([0,1]^{\mathbb{N}})$, we have an automorphism of $([0,1]^{\mathbb{N}})^{\mathbb{Z}}$ denoted by $f^{\mathbb{Z}}$  such that $f^{\mathbb{Z}}(x)(n)=f(x(n))$. Note that  $G=\{f^{\mathbb{Z}}|f\in\rm{Aut}([0,1]^{\mathbb{N}})\}$ is a closed subgroup of $\rm{Aut}(([0,1]^{\mathbb{N}})^{\mathbb{Z}})$ thus a Polish group. We write $E^G_D$ to denote the equivalence relation induced by the action of $G$ on $D$. We will prove that the conjugacy relation of minimal compact systems is Borel reducible to $E_D^G$. This will prove the theorem since $E^G_D$ is induced by a Polish group action.

                  First, we embed $([0,1]^{\mathbb{N}})^{\mathbb{Z}}$ to $[0,1]^{\mathbb{N}}$ as a $Z$-set by $h$. 

                  Now for a minimal subshift $(X,S)$ of $(([0,1]^{\mathbb{N}})^{\mathbb{Z}},S)$, we embed $X$ into $([0,1]^{\mathbb{N}})^{\mathbb{Z}}$ by $\iota$, where
                  $$
                  \iota(x)(n)=h(S^nx),\, \forall n\in \mathbb{Z}
                  $$
                  It is easy to check that $(\iota(X),S)$ is topologically conjugate to $(X,S)$ by $\iota$, thus the topological conjugacy of $(\iota(Y),S)$ and $(\iota(X),S)$ implies the conjugacy of $(X,S)$ and $(Y,S)$.

                  On the other hand, suppose $(X,S)$ and $(Y,S)$ are conjugate by $f$. Then $h f h^{-1}$ is a homeomorphism between $h(X)$ and $h(Y)$. Since both of them are $Z$-sets, this homeomorphism extends to an auto-homeomorphism of $[0,1]^{\mathbb{N}}$, denoted by $g$. We know that $g|_{h(X)}=h f h^{-1}$. We consider $g^{\mathbb{Z}}(\iota(X))$, each $x\in X$, 
                  $$
                  g^{\mathbb{Z}}(\iota(x))(n)=gh(S^nx)=hfh^{-1}h(S^nx)=h(S^nf(x)).
                  $$
                  Thus we have  $g^{\mathbb{Z}}(\iota(x))=\iota(f(x))\in \iota(Y)$ which implies $g^{\mathbb{Z}}(\iota(X))\subset \iota(Y)$. By a symmetric argument we get the other inclusion, thus, $g^{\mathbb{Z}}(\iota(X))= \iota(Y)$. Thus, $\iota$ is a Borel reduction.
               \end{proof}

            Now we compute the Borel complexity of the relation of conjugacy of pointed minimal compact systems.
               \begin{theorem}
                   The conjugacy relation of pointed minimal compact systems is a $\mathbf{\Pi}^0_3$ equivalence relation but not Borel reducible to any $\mathbf{\Sigma}^0_3$ equivalence relation, in particular, it is a  $\mathbf{\Pi}_3^0$-complete set.
               \end{theorem}
               \begin{proof}
                   Note that the set of all pointed minimal compact systems is a $G_\delta$ subset of $\mathcal{K}(([0,1]^{\mathbb{N}})^\mathbb{Z})\times ([0,1]^{\mathbb{N}})^\mathbb{Z}$. Let $d$ be a compatible metric on $([0,1]^{\mathbb{N}})^\mathbb{Z}$. For two pointed minimal compact systems $(X,S,z)$ and $(X',S,z')$, suppose $(X,S,z)$ is conjugate to $(X',S,z')$ and let $f$ be a conjugacy map. Note that since $f$ and $f^{-1}$ are uniformly continuous, we have
                   \begin{equation*}\label{eq9.2}
                     \forall n\ge1\exists m\ge1 \forall i,j\in\mathbb{Z} ((d(S^iz,S^jz)<\frac{1}{m} \rightarrow d(S^if(z),S^jf(z))\le\frac{1}{n})\wedge(d(S^if(z),S^jf(z))<\frac{1}{m}\rightarrow d(S^iz,S^jz)\le\frac{1}{n}).
                   \end{equation*}
                   This means that if two pointed  minimal compact systems $(X,S,z)$ and $(X',S,z')$ are conjugate then the following condition holds:
                   \begin{equation}\label{eq9.1}
                     \forall n\ge1\exists m\ge1 \forall i,j\in\mathbb{Z} ((d(S^iz,S^jz)<\frac{1}{m} \rightarrow d(S^iz',S^jz'\le\frac{1}{n})\wedge(d(S^iz',S^jz')<\frac{1}{m}\rightarrow d(S^iz,S^jz)\le\frac{1}{n}).
                   \end{equation}
                   
                    On the other hand, if (\ref{eq9.1}) holds, then by Lemma \ref{min}, $(X,S,z)$ is conjugate to $(X',S,z')$. So by the form of (\ref{eq9.1}), the conjugacy relation of pointed minimal compact systems is a $\mathbf{\Pi}^0_3$ equivalence relation.
                   
                   It is well-known that $=_\mathbb{R}^{+}$ is not Borel reducible to any $\mathbf{\Sigma}^0_3$ equivalence relation \cite[Theorem 8.5.2 and Lemma 8.5.4]{Gao}. By Kaya's result \cite[Theorem 2]{Kaya}, $=_{\mathbb{R}}^+$ is Borel-reducible to the conjugacy relation of pointed Cantor minimal systems thus to the conjugacy relation of pointed minimal compact systems. So the conjugacy relation of pointed minimal compact systems is not Borel reducible to any $\mathbf{\Sigma}^0_3$ equivalence relation.
               \end{proof}

               \section{Remarks and questions}

               The exact complexity of the conjugacy relation of minimal compact systems is still open. Sabok made the following conjecture:

               \vspace{0.5em}

              \noindent \textbf{Conjecture 8.1} \textit{{\rm (Sabok)} The conjugacy relation of minimal compact systems is a complete orbit equivalence relation.}

               \vspace{0.5em}

             When studying topological types of sequences, it is also natural to ask the following question:

                 \vspace{0.5em}

              \noindent \textbf{Question 8.2} Let $M_1$ and $M_2$ be two manifolds with different dimensions. Do $E_{\rm tt}(M_1)$ and $E_{\rm tt}(M_2)$ have different complexity? In particular, is $E_{\rm tt}([0,1])<_B E_{\rm tt}([0,1]^{\mathbb{N}})$?

               \vspace{0.5em}
               \,
               \,

                \noindent \textbf{Acknowledgements.} We would like to thank Su Gao and Marcin Sabok for their valuable advice. Theorem \ref{Turbulent} was suggested by Marcin Sabok and we have learnt recently that a similar argument also appears in the work of Longyun Ding and Kai Gu \cite{DG}. The first author is partly funded by NSFC grants 12250710128 and 12271263. The second author is partly funded by NSERC Discovery Grant RGPIN-2020-05445.

\end{document}